\documentclass[hidelinks,onefignum,onetabnum]{siamart251216}

\usepackage[T1]{fontenc}
\usepackage{lmodern}
\usepackage[protrusion=true,expansion=false]{microtype}
\usepackage{amsmath,amssymb,mathtools,mathrsfs}
\usepackage{enumitem,needspace}

\allowdisplaybreaks
\DeclareMathOperator{\Div}{div}

\DeclareMathOperator{\Ran}{Ran}
\DeclareMathOperator{\Ker}{Ker}
\DeclareMathOperator{\Span}{span}
\newcommand{\spec}{\sigma}

\newcommand{\R}{\mathbb{R}}
\newcommand{\ud}{\,\mathrm{d}}
\newcommand{\cE}{\mathcal{E}}
\newcommand{\cH}{\mathcal{H}}
\newcommand{\cQ}{\mathcal{Q}}
\newcommand{\Id}{\mathrm{Id}}
\newcommand{\eps}{\varepsilon}

\newsiamremark{hypothesis}{Hypothesis package}
\newsiamremark{remark}{Remark}
\crefname{hypothesis}{hypothesis package}{hypothesis packages}

\headers{Stability of non-isentropic gaseous stars}{Z. Lin, Y. Wang, and H. Zhu}
\title{Stability of non-isentropic gaseous stars}
\author{Zhiwu Lin\thanks{School of Mathematical Sciences, Fudan University,
Shanghai 200433, China (\email{zwlin@fudan.edu.cn}).}
\and Yucong Wang\thanks{School of Mathematics and Computational Science,
Xiangtan University, Xiangtan, Hunan 411105, China
(\email{yucongwang666@163.com}).}
\and Hao Zhu\thanks{School of Mathematics, Nanjing University, Nanjing,
Jiangsu 210093, China (\email{haozhu@nju.edu.cn}).}}

\ifpdf
\hypersetup{
 pdftitle={Stability of non-isentropic gaseous stars},
 pdfauthor={Zhiwu Lin, Yucong Wang, and Hao Zhu},
 pdfkeywords={Euler--Poisson system, gaseous stars, non-isentropic flow,
 turning-point principle, Schwarzschild criterion, Hamiltonian PDE}}
\fi

\begin{document}

\begingroup
\hbadness=10000
\maketitle
\endgroup

\begin{abstract}
We study the linear stability of compactly supported static spherical equilibria of
the non-isentropic Euler--Poisson system.  In the Schwarzschild-stable case,
the linearized equations are realized as a separable Hamiltonian system on
weighted spaces adapted to the physical vacuum.  We prove stability against
non-radial perturbations and show that the algebraic dimension of the radial unstable subspace is
the Morse index of a density quadratic form subject only to the mass
 constraint.  Thus the infinitely many linearized entropy constraints reduce,
 for radial perturbations, to one mass constraint and an explicit entropy
 reconstruction.  We apply this criterion to two entropy prescriptions.  No
 smallness condition is imposed on a fixed entropy--density relation: under
 the stated effective-pressure and branch hypotheses, a simple mass maximum
 is not a stability transition.  By contrast, for a small entropy distribution
 fixed as a function of enclosed mass, the turning-point principle holds.  In the
Schwarzschild-unstable case, we construct the self-adjoint velocity operator
from its quadratic form, prove a strictly negative spectral bottom and sharp
exponential growth, and reconstruct energy solutions of the original
first-order linearized system.
\end{abstract}



\section{Introduction}\label{sec:introduction}

The motion of a self-gravitating, non-isentropic gas is described by the
Euler--Poisson system
\begin{equation}\label{eq:EP}
\left\{
\begin{aligned}
 &\partial_t\rho+\Div(\rho u)=0,\\
 &\rho\bigl(\partial_tu+u\cdot\nabla u\bigr)+\nabla P(S,\rho)
      +\rho\nabla V=0,\\
 &\partial_tS+u\cdot\nabla S=0,\\
 &
V(x)=-\int_{\R^3}\frac{\rho(y)}{|x-y|}\ud y.
\end{aligned}
\right.
\end{equation}
Here $\rho\geq0$ is the density, $u$ the velocity, $S$ the specific entropy,
$P=P(S,\rho)$ the pressure, and $V$ the gravitational
potential.  Units are chosen so that the gravitational constant is one.
Throughout, the Newton operator is normalized by
\[
 ((-\Delta)^{-1}h)(x)=\frac1{4\pi}\int_{\R^3}\frac{h(y)}{|x-y|}\ud y,
 \qquad V_h=-4\pi(-\Delta)^{-1}h.
\]
We consider compactly supported static
spherical solutions
\[
 (\rho_0,S_0,u_0)=(\rho_0(r),S_0(r),0),\qquad
 \operatorname{supp}\rho_0=\overline{B_R},\qquad r=|x|.
\]
Here $B_R=\{x\in\R^3:|x|<R\}$ denotes the ball of radius $R$.
When $\rho_0$ is strictly decreasing, the equilibrium entropy can be written
as $S_0=f(\rho_0)$.  The coefficient
\begin{equation}\label{eq:G-intro}
 \mathfrak G(r)
 =-\frac{P_S(S_0(r),\rho_0(r))f'(\rho_0(r))}{\rho_0(r)}
\end{equation}
has the sign of the classical Schwarzschild discriminant.  Indeed,
\begin{equation}\label{eq:classical-discriminant-relation}
 \frac{\ud}{\ud r}P(S_0,\rho_0)-P_\rho(S_0,\rho_0)\rho_0'
 =-\rho_0\rho_0'\mathfrak G.
\end{equation}
Thus $\mathfrak G>0$ is the restoring sign, while $\mathfrak G<0$ on a
layer is the convectively unstable sign.  It is also the sign of the squared
Brunt--V\"ais\"al\"a frequency, up to the usual positive normalization; see
\cite{LedouxWalraven1958,Cox1980,Unno1989,Aerts2010}.  A modern operator
treatment of adiabatic stellar oscillations was given by Makino
\cite{Makino2023}.

For barotropic isentropic  nonrotating stars, Lin and Zeng \cite{LinZeng2022} proved that
the radial instability index is determined by the mass--radius curve.  In
particular, stability can change only at an extremum of the mass.  Their proof
combines a separable Hamiltonian formulation with the classical variational
principles for stellar oscillations
\cite{Eddington1918,Chandrasekhar1964,ChandrasekharLebovitz1964}.
For isentropic rotating stars, the answer depends on the comparison family: fixed
angular velocity and fixed Lagrangian angular-momentum distribution lead to
different turning-point laws \cite{LinWang2023}.

The non-isentropic problem has an additional difficulty.  The entropy
equation gives the conserved functionals
\begin{equation}
 \mathcal C_g(\rho,S)=\int_{\R^3}\rho g(S)\ud x,
 \qquad g\in C^1(\R),
\end{equation}
and hence infinitely many linear constraints on dynamically accessible
density--entropy perturbations.  We use the Casimirs whose linearizations
are continuous on the phase space, as specified in
\cref{prop:Casimir-accessibility}.  In the radial class these constraints are
equivalent to conservation of mass together with an explicit reconstruction
of the entropy perturbation.  The stability problem consequently reduces to
a density quadratic form on a codimension-one space.

Existence and stability of non-isentropic equilibria have been considered in
several settings; see
\cite{DengLiuYangYao2002,DengGuo2003,DengYang2006,DengBaXie2010,
Wu2015,Yuan2022,HongLuoZhu2018,DuanLuoWang2025}.  The classical works of
Lebovitz \cite{Lebovitz1965,LebovitzOnset1965,Lebovitz1966} established the
role of Schwarzschild's criterion for spherical gaseous masses; see also
\cite{KanielKovetz1967,Rosencrans1969} for related treatments of the criterion.  The static
criterion of Bisnovatyi-Kogan and Blinnikov
\cite{BisnovatyiKoganBlinnikov1974} also stresses that a comparison sequence
must preserve the Lagrangian distributions of advected quantities.  The
results below give a Hamiltonian and constrained-inertia explanation of this
distinction.

We first consider the Schwarzschild-stable case $\mathfrak G>0$.  Introduce
\[
 \eta=\delta\rho,\qquad
 \zeta=\delta\rho-\frac{\delta S}{f'(\rho_0)}.
\]
The linearized equations take the separable Hamiltonian form
\begin{align}\label{separable-Hamiltonian-entropy}
 \partial_t\binom{(\eta,\zeta)}{u}
 =\begin{pmatrix}0&B\\-B'&0\end{pmatrix}
  \begin{pmatrix}\mathbb L&0\\0&A\end{pmatrix}
  \binom{(\eta,\zeta)}{u},
\end{align}
where operators $A$, $\mathbb L$, $B$ are defined by \eqref{eq:A-and-mathbbL} and \eqref{eq:B-maximal}, and $B'$
denotes the dual operator of $B$.

Let $n^u$ be the algebraic dimension of the unstable spectral
subspace.
The
separable Hamiltonian theorem of \cite{LinZeng2022} then gives
\[
 n^u=n^-\!\left(\mathbb L\big|_{\overline{\Ran B}}\right).
\]
Here $n^-(T)$ denotes the Morse index of $T$: the maximal
dimension of a subspace on which the quadratic form of $T$ is negative
definite.

For a radial density perturbation set
\[
 I_\eta(r)=4\pi\int_0^r s^2\eta(s)\ud s,
 \qquad
 \mathsf H_0=\left\{\eta:\int_{B_R}\eta\ud x=0\right\}.
\]
The entropy constraints are equivalent to
\[
 \delta S(r)
 =\frac{f'(\rho_0)\rho_0'}{4\pi r^2\rho_0}I_\eta(r),
 \qquad \eta\in \mathsf H_0.
\]
This reconstruction is proved as
\eqref{eq:radial-entropy-reconstruction}.
The reduced form is
\begin{equation}\label{eq:reduced-form-intro}
\begin{split}
 \cQ_0[\eta]
={}&\left\langle
 \bigl(\Psi''(\rho_0)-4\pi(-\Delta)^{-1}\bigr)\eta,\eta
 \right\rangle\\
&+\int_{B_R}\mathfrak G(r)
 \left|\eta-\frac{\rho_0'}{4\pi r^2\rho_0}I_\eta(r)\right|^2\ud x,
\end{split}
\end{equation}
where $\Psi''$ is the enthalpy weight for the composite pressure
$P(f(\rho),\rho)$, $(-\Delta)^{-1}$ denotes the Newton
potential, and $\langle\cdot,\cdot\rangle$ is the unweighted integral
duality pairing.

\begin{theorem}[Linear stability criterion]
\label{thm:intro-main-index}
Assume the pressure and Schwarzschild stable equilibrium hypotheses in
\cref{hyp:pressure,hyp:stable-equilibrium}. Then the linearized system \eqref{separable-Hamiltonian-entropy} generates a
strongly continuous group with a finite-dimensional unstable subspace. Moreover, the non-isentropic star  $(\rho_{0},S_0,0)$
is spectrally stable against non-radial perturbations.  In the radial subspace, the number of unstable eigenvalues, counting its algebraic multiplicity,  is
\[
 n^u_{\mathrm{rad}}=n^-(\cQ_0|_{\mathsf H_0}).
\]

\end{theorem}

We next compare two families of equilibria parameterized by the central density $\mu$.
For a fixed relation $S_\mu=f(\rho_\mu)$, the family
tangent is normal to the mass constraint for the barotropic part of
\eqref{eq:reduced-form-intro}, but not for the full form: the positive
Schwarzschild term lifts the barotropic neutral direction at a mass maximum.
For fixed Lagrangian entropy,
\[
 S_{\mu,\kappa}(r)=\kappa\overline S(m_{\mu,\kappa}(r)),
 \qquad
 m_{\mu,\kappa}(r)=4\pi\int_0^r s^2\rho_{\mu,\kappa}(s)\ud s,
\]
the tangent is normal for the full reduced Hessian.  This gives  opposite
turning-point conclusions informally.

\begin{theorem}[Two entropy prescriptions]
\label{thm:intro-two-families}
\begin{enumerate}[label=(\roman*)]
\item Let $S=f(\rho)$, with no smallness assumption on $f$, and suppose the
effective pressure satisfies the hypotheses of
\cref{lem:fixed-relation-existence}.  The fixed-relation equilibria exist for
all sufficiently small central densities.  On any compact continuation
satisfying $\mathrm{(E_s)}$ and $\mathrm{(F)}$, every simple mass extremum
obeying the nondegeneracy and index assumptions of
\cref{thm:fixed-relation-no-transition} is not a radial stability transition.
\item Under the factorized-pressure and comparison-branch hypotheses of
\cref{thm:lagrangian-small-entropy-tpp}, if
$S=\kappa\overline S(m)$ with $\overline S'>0$ and $\kappa>0$ is sufficiently
small, the unstable radial
dimension is zero when $M'\geq0$ and one when $M'<0$.  The constrained
kernel at the simple mass maximum is one-dimensional, so the stability
transition occurs precisely there.
\end{enumerate}
The regularity of $f$ and $\overline S$ is specified in
\cref{lem:fixed-relation-existence,lem:lagrangian-family-existence}.
\end{theorem}

The simple maximum and the spectral properties of the continued
fixed-relation branch are explicit hypotheses; they do not follow from the
vacuum exponent.  The proofs show that fixing a function of density and
fixing a function of enclosed mass lead to different Hessian normal
identities.

Finally, suppose that $\mathfrak G<0$ on an interior open set.  The
hydrostatic equation still forces the effective compressibility to be
positive for the decreasing equilibria considered here
(see \eqref{eq:effective-pressure-positive} and
\eqref{eq:effective-pressure-positive-unstable}).  The velocity
operator is associated with the closed form
\begin{align*}
 \mathfrak q[u]
 &=\int_{B_R}\left(\mathfrak aD_u^2-2\mathfrak GD_ua_u
                      +\mathfrak Ga_u^2\right)\ud x
     -4\pi\langle(-\Delta)^{-1}D_u,D_u\rangle,\\
 D_u&=\Div(\rho_0u),\qquad a_u=\nabla\rho_0\cdot u,
 \qquad \mathfrak a=P_\rho(S_0,\rho_0)/\rho_0.
\end{align*}
Its domain is $\{u\in L^2_{\rho_0}:D_u\in L^2_{\mathfrak a}\}$.
The expanded expression is the definition; each of its terms is finite
under the coefficient bounds stated below.

\begin{theorem}[Schwarzschild instability]
\label{thm:intro-convective}
Under the coefficient and form-domain assumptions in
\cref{hyp:unstable-form}, the self-adjoint operator $\widetilde L$ associated with the
velocity form has $\inf\spec(\widetilde L)<0$.  The velocity wave equation
has exponentially growing data with zero initial time derivative, with rates
arbitrarily close to $\sqrt{-\inf\spec(\widetilde L)}$.  These data reconstruct
energy solutions of the original first-order linearized Euler--Poisson
system.
\end{theorem}

The classical papers of Lebovitz are close predecessors of the
non-radial and Schwarzschild parts of the present work.  Using
Chandrasekhar's displacement principle, Lebovitz proved that the
non-radial characteristic frequencies are real when the Schwarzschild
discriminant is positive and identified the translational kernel in the
$\ell=1$ sector \cite{Lebovitz1965}.  Passing from a normal-mode analysis
to evolution of general initial data requires a specified operator domain
and phase space.  The Hamiltonian group in
\cref{thm:intro-main-index} provides this operator-theoretic step in the phase
space used here, while retaining the neutral translations and stationary
velocities.  Lebovitz subsequently characterized neutral convective modes
when the discriminant vanishes on a layer \cite{LebovitzOnset1965} and
constructed a destabilizing displacement when it is negative
\cite{Lebovitz1966}.  The form construction behind
\cref{thm:intro-convective} is the Eulerian counterpart of the latter
argument; it supplies spectral growth rates and a weak reconstruction of
the first-order variables in the spaces specified here.  Self-adjoint
displacement realizations for nonbarotropic adiabatic oscillations are
already available in \cite{Makino2023}.  Our general-pressure Hamiltonian
formulation and exact radial entropy-Casimir reduction address different
parts of the stability problem.

The results are linear and concern perturbations in the reference fluid
region.  They do not imply nonlinear stability of the physical-vacuum free
boundary.  The new points are the closed Hamiltonian realization, the exact
radial reduction of the entropy constraints, the different inertia
identities for the two entropy prescriptions, and the closed-form treatment
of Schwarzschild instability.

\Cref{sec:hypotheses} states the assumptions and stability notions.
\Cref{sec:hamiltonian} establishes the Hamiltonian formulation and its
functional setting.  The non-radial and radial reductions are proved in
\cref{sec:nonradial,sec:radial}.  The two families are considered in
\cref{sec:fixed-relation,sec:fixed-lagrangian}.  Schwarzschild instability is
proved in \cref{sec:convective}.

\section{Hypotheses, notation, and stability notions}\label{sec:hypotheses}

We state the assumptions in packages so that every later result can refer to
the exact hypotheses it uses.  Stronger, family-specific assumptions are added
only where parameter differentiation or perturbation from an isentropic
family is required.

\Needspace{7\baselineskip}
\subsection{Pressure and vacuum asymptotics}

\begin{hypothesis}[Pressure package $\mathrm{(P)}$]\label{hyp:pressure}
The pressure satisfies
\[
 P\in C^3(\R\times(0,\infty))\cap C^1(\R\times[0,\infty)),
 \qquad P(S,0)=0,
\]
and
\[
 P_\rho(S,\rho)>0,
 \qquad P_S(S,\rho)>0
 \quad (\rho>0).
\]
There are $\gamma_0\in(6/5,2)$ and $\alpha_0>0$ such that, on every compact
entropy interval $J$, one can find $K\in C^2(J)$ with
$\inf_JK>0$ and
\begin{equation}\label{eq:uniform-pressure-asymptotics}
 \partial_S^j\partial_\rho^k
 \bigl(P(S,\rho)-K(S)\rho^{\gamma_0}\bigr)
 =O_J\!\left(\rho^{\gamma_0-k+\alpha_0}\right),
 \qquad j+k\leq2,\quad k\leq2,
\end{equation}
as $\rho\downarrow0$, uniformly for $S\in J$.  All third derivatives used
below satisfy the explicit bounds
\begin{equation}
 \bigl|\partial_S^j\partial_\rho^kP(S,\rho)\bigr|
 \leq C_J\rho^{\gamma_0-k},
 \qquad j+k=3,
 \quad 0<\rho\leq1,
\end{equation}
uniformly for $S\in J$.  In particular,
$P_S(S,\rho)/\rho^2$ is integrable at vacuum.
\end{hypothesis}

Define the internal-energy density by the vacuum normalization
\begin{equation}\label{eq:Phi-definition}
 \Phi(S,\rho)
 =\rho\int_0^\rho\frac{P(S,\tau)}{\tau^2}\ud\tau,
 \qquad \Phi(S,0)=\Phi_\rho(S,0)=0.
\end{equation}
Then
\begin{equation}\label{eq:thermodynamic-identities}
 \rho\Phi_\rho-\Phi=P,
 \qquad
 \Phi_{\rho\rho}=\frac{P_\rho}{\rho},
 \qquad
 \partial_\rho\!\left(\frac{\Phi_S}{\rho}\right)
   =\frac{P_S}{\rho^2},
\end{equation}
and hence
\begin{equation}\label{eq:pressure-gradient-identity}
 \frac{\nabla P(S,\rho)}{\rho}
 =\nabla\Phi_\rho(S,\rho)-\frac{\Phi_S(S,\rho)}{\rho}\nabla S.
\end{equation}
These formulas follow directly by differentiating
\eqref{eq:Phi-definition}; the uniform remainder in
\eqref{eq:uniform-pressure-asymptotics} justifies differentiation at the
vacuum endpoint.

\subsection{Equilibria and the Schwarzschild coefficient}

\begin{hypothesis}[Equilibrium package $\mathrm{(E_s)}$]\label{hyp:stable-equilibrium}
The triple $(\rho_0,S_0,0)$ is a compactly supported radial equilibrium with
support $\overline{B_R}$, total mass
\[
 M=4\pi\int_0^Rr^2\rho_0(r)\ud r>0,
\]
and
\begin{equation}\label{eq:equilibrium-density-regularity}
\begin{gathered}
 \rho_0\in C^2([0,R))\cap C^1([0,R]),
 \qquad \rho_0(0)>0,\\
 \rho_0'(r)<0\quad(0<r<R),
 \qquad \rho_0(R)=0.
\end{gathered}
\end{equation}
There is
$f\in C^1([0,\rho_0(0)])\cap C^2((0,\rho_0(0)))$
such that $S_0=f(\rho_0)$. Let
$$ \widetilde P(\rho_0):=P(f(\rho_0),\rho_0)$$
 and assume that
\begin{equation}\label{eq:effective-pressure-positive}
 \widetilde P'(\rho_0)
 :=P_\rho(S_0,\rho_0)+P_S(S_0,\rho_0)f'(\rho_0)>0.
\end{equation}
The Schwarzschild-stable package additionally requires
\begin{equation}\label{eq:Schwarzschild-stable}
 \mathfrak G(r)
 :=-\frac{P_S(S_0,\rho_0)f'(\rho_0)}{\rho_0}>0
 \quad\hbox{for almost every }0<r<R,
\end{equation}
with $\mathfrak G^{-1}\in L^1_{\mathrm{loc}}(B_R)$; endpoint vanishing at
$r=0$ or $r=R$ is allowed.  The coefficient $\mathfrak G$ is
the Schwarzschild coefficient \eqref{eq:G-intro}.  We require
\begin{equation}\label{eq:weight-comparability}
 \begin{gathered}
 0<\mathfrak G(r)\leq C\rho_0^{\gamma_0-1}(r)
 \quad\text{for $r$ near $R$},\\
 \lim_{\rho\downarrow0}\frac{\widetilde P'(\rho)}{\rho^{\gamma_0-1}}
 =\gamma_0K(f(0))>0.
 \end{gathered}
\end{equation}
\end{hypothesis}

Because $P_S>0$, the almost-everywhere sign in
\eqref{eq:Schwarzschild-stable}, together with
$\mathfrak G^{-1}\in L^1_{\mathrm{loc}}(B_R)$, actually implies
\begin{equation}\label{eq:G-strict-interior}
 \mathfrak G(r)>0\qquad(0<r<R).
\end{equation}
Indeed, on every compact shell contained in $(0,R)$ the regularity of
$P,f$, and $\rho_0$ makes $\mathfrak G$ locally Lipschitz.  Its
almost-everywhere positivity and continuity first give $\mathfrak G\geq0$.
If $\mathfrak G(r_0)=0$ at an interior point, then
$0\leq\mathfrak G(r)\leq C|r-r_0|$ near $r_0$, and hence
\[
 \int_{r_0-\delta}^{r_0+\delta}
 \frac{r^2}{\mathfrak G(r)}\ud r
 \geq c\int_{r_0-\delta}^{r_0+\delta}
 \frac{\ud r}{|r-r_0|}=\infty,
\]
contrary to the reciprocal-integrability hypothesis.  Thus, since
$P_S>0$ and $\rho_0>0$ in the open ball,
\eqref{eq:G-strict-interior} is equivalent to $f'<0$ throughout the
interior of the density range; $f'$ may still vanish at its endpoints.

The effective pressure potential is
\begin{equation}\label{eq:Psi-definition}
 \Psi(\rho)=\rho\int_0^\rho\frac{\widetilde P(\tau)}{\tau^2}\ud\tau,
 \qquad
 \Psi(0)=\Psi'(0)=0,
 \qquad
 \Psi''(\rho)=\frac{\widetilde P'(\rho)}{\rho}.
\end{equation}

\begin{lemma}\label{lem:physical-vacuum}
Under $\mathrm{(P)}$ and $\mathrm{(E_s)}$, there are $r_1<R$ and positive
constants $c,C$ such that, for $r_1<r<R$,
\begin{equation}\label{eq:physical-vacuum}
 c(R-r)^{1/(\gamma_0-1)}
 \leq\rho_0(r)\leq
 C(R-r)^{1/(\gamma_0-1)},
\end{equation}
and
\begin{equation}\label{eq:physical-vacuum-derivative}
 c(R-r)^{(2-\gamma_0)/(\gamma_0-1)}
 \leq-\rho_0'(r)\leq
 C(R-r)^{(2-\gamma_0)/(\gamma_0-1)}.
\end{equation}
\end{lemma}

\begin{proof}
The hydrostatic equation is
\[
 \widetilde P'(\rho_0)\rho_0'=-\rho_0V_0',
 \qquad
 V_0'(r)=\frac{4\pi}{r^2}\int_0^rs^2\rho_0(s)\ud s.
\]
Thus $V_0'(R)=M/R^2>0$.  By \eqref{eq:Psi-definition},
\[
 \Psi'(\rho_0(r))-\Psi'(0)=V_0(R)-V_0(r).
\]
The right-hand side is comparable to $R-r$, while
\eqref{eq:uniform-pressure-asymptotics} and
\eqref{eq:weight-comparability} give
$\Psi'(\rho_0)-\Psi'(0)\asymp\rho_0^{\gamma_0-1}$ (here and
below, $f\asymp g$ means $cf\leq g\leq Cf$ for positive constants
$c,C$).  This proves
\eqref{eq:physical-vacuum}.  Differentiating the preceding identity and using
$\Psi''(\rho_0)\asymp\rho_0^{\gamma_0-2}$ gives
\eqref{eq:physical-vacuum-derivative}.
\end{proof}

\subsection{Weighted spaces and duality}

For a positive weight $w$ on $B_R$, set
\begin{equation}\label{eq:weighted-space}
 L_w^2(B_R)=\left\{h:\int_{B_R}w|h|^2\ud x<\infty\right\},
 \qquad
 \|h\|_{L_w^2}^2=\int_{B_R}w|h|^2\ud x.
\end{equation}
We identify $(L_w^2)^*$ with $L_{1/w}^2$ through the unweighted pairing
$\langle h,k\rangle=\int h k\ud x$.  In the Schwarzschild-stable regime,
\begin{equation}\label{eq:phase-spaces}
\begin{gathered}
 X_1=L_{\Psi''(\rho_0)}^2(B_R),
 \qquad X_2=L_{\mathfrak G}^2(B_R),\\
 X=X_1\times X_2,
 \qquad Y=L_{\rho_0}^2(B_R;\R^3).
\end{gathered}
\end{equation}

\subsection{Stability notions and parameterized families}
For a family with moving support, spectral continuity must be formulated on a
fixed space.  We use the mass-preserving pullback
\begin{equation}\label{eq:mass-preserving-pullback}
 (\mathcal U_\mu\eta)(s)=R_\mu^3\eta(R_\mu s),
 \qquad 0<s<1.
\end{equation}
Here $R_\mu$ is the first vacuum radius of the equilibrium with
central density $\mu$.

\begin{hypothesis}[Family package $\mathrm{(F)}$]\label{hyp:family}
On the compact parameter set under consideration, the mass and radius are
$C^2$ in the family parameters.  The equilibrium profiles are $C^2$ in
those parameters on compact subsets of the positive-density interior.
The first derivative $v_\mu=\partial_\mu\rho_\mu$ at fixed physical radius,
extended by zero, belongs to the relevant weighted density space, and
$\int v_\mu\ud x=M_\mu'$.  No second derivative of the zero extension in
that weighted norm is assumed.  The pullbacks
\eqref{eq:mass-preserving-pullback} identify the moving spaces with one fixed
radial Hilbert space.  In this realization the quadratic forms are
norm-continuous.  Each form is the sum of a uniformly coercive bounded
symmetric form and a compact form and has finite Morse index.
\end{hypothesis}

The general family results explicitly assume $\mathrm{(F)}$.  The
fixed-relation theorem assumes it on the compact branch under consideration;
the construction in \cref{sec:fixed-lagrangian} verifies it for the
small fixed-Lagrangian-entropy family.  Notice that
none of the inertia arguments requires differentiability in operator norm of
the singular reconstruction map; norm continuity of the pulled-back forms is
the exact property used below.

\begin{definition}[Turning-point principle]\label{def:TPP}
Let a $C^2$ central-density family have a simple mass extremum at $\mu_*$,
meaning $M'(\mu_*)=0$ and $M''(\mu_*)\ne0$.  We
call $\mu_*$ a \emph{radial stability transition} if the constrained form has
a one-dimensional kernel there, is nondegenerate immediately on either side,
and its Morse index changes by one across $\mu_*$.  The family satisfies the
\emph{first turning-point principle} on an interval if its first radial
stability transition is its first simple mass extremum.  This definition
concerns the growing-mode count; a neutral mode at the extremum need not give
a uniformly bounded center group.
\end{definition}

\section{Hamiltonian realization of the linearized equations}
\label{sec:hamiltonian}

Throughout this section we assume the pressure package $\mathrm{(P)}$ and the
Schwarzschild-stable equilibrium package $\mathrm{(E_s)}$.  All differential
expressions below are initially understood in the sense of distributions on
$B_R$.  This convention is important because the coefficient
$\nabla\rho_0/\rho_0$ is singular at the vacuum boundary.

\subsection{Linearized variables and the energy blocks}

Let $\eta$, $s$, and $u$ denote, respectively, the density, entropy, and
velocity perturbations.  Linearizing \eqref{eq:EP} at
$(\rho_0,S_0,0)$ gives
\begin{equation}\label{eq:linearized-EP}
\left\{
\begin{aligned}
 &\partial_t\eta+\Div(\rho_0u)=0,\\
 &\partial_t s+u\cdot\nabla S_0=0,\\
 &\partial_tu
 +\frac1{\rho_0}\nabla\bigl(P_\rho(S_0,\rho_0)\eta
                         +P_S(S_0,\rho_0)s\bigr)
-\frac{\eta}{\rho_0^2}\nabla P(S_0,\rho_0)+\nabla V_\eta=0,\\
 &
 V_\eta=-4\pi(-\Delta)^{-1}\eta.
\end{aligned}
\right.
\end{equation}
Introduce the density--entropy variables
\begin{equation}\label{eq:eta-zeta-variables}
 z=\binom{\eta}{\zeta},
 \qquad
 \zeta=\eta-\frac{s}{f'(\rho_0)}.
\end{equation}
Since $S_0=f(\rho_0)$, the first two equations in
\eqref{eq:linearized-EP} imply
\begin{equation}\label{eq:z-kinematic-equation}
 \partial_tz
 =\binom{-\Div(\rho_0u)}
 {-\Div(\rho_0u)+(\nabla\rho_0/\rho_0)\cdot(\rho_0u)}.
\end{equation}

Set
\begin{equation}\label{eq:a0-and-weights}
 a_0=\frac{\nabla\rho_0}{\rho_0},
 \qquad
 w_1=\Psi''(\rho_0),
 \qquad
 w_2=\mathfrak G.
\end{equation}
On the spaces in \eqref{eq:phase-spaces}, define
\begin{equation}\label{eq:A-and-mathbbL}
 Au=\rho_0u:Y\longrightarrow Y^*,
 \qquad
 \mathbb L=
 \begin{pmatrix}
  L_{11}&0\\0&L_{22}
 \end{pmatrix}:X\longrightarrow X^*,
\end{equation}
where
\begin{equation}\label{eq:L11-L22}
 L_{11}=w_1-4\pi(-\Delta)^{-1},
 \qquad
 L_{22}=w_2.
\end{equation}
Here $(-\Delta)^{-1}\eta$ denotes the Newton potential restricted to $B_R$.
The Riesz map $A$ is an isometric isomorphism from $Y$ to $Y^*$, and
$L_{22}$ is the Riesz map from $X_2$ to $X_2^*$.

Using the hydrostatic equation and the definitions of $\Psi$ and
$\mathfrak G$, the momentum equation in \eqref{eq:linearized-EP} becomes
\begin{equation}\label{eq:linearized-momentum-separated}
 \partial_tu
 =-\nabla(L_{11}\eta)-\bigl(\nabla+a_0\bigr)(w_2\zeta).
\end{equation}
This identity also follows from the energy--Casimir calculation in
\cref{lem:energy-casimir-hessian} below.

\subsection{The singular first-order operator}

Define the maximal distributional operator
\begin{equation}\label{eq:B-maximal}
 B:D(B)\subset Y^*\longrightarrow X,
 \qquad
 Bq=\binom{B_1q}{B_2q}
 =\binom{-\Div q}{-\Div q+a_0\cdot q},
\end{equation}
with graph domain
\begin{equation}\label{eq:B-domain}
 D(B)=\left\{q\in L^2_{1/\rho_0}(B_R;\R^3):
 \begin{array}{l}
 -\Div q\in X_1,\\[-2pt]
 -\Div q+a_0\cdot q\in X_2
 \end{array}
 \right\}.
\end{equation}
The two conditions in \eqref{eq:B-domain} are distributional conditions in
the open ball, with the weighted spaces understood as in
\eqref{eq:weighted-space}.  No pointwise trace is imposed at
$\partial B_R$.

\begin{lemma}\label{lem:B-closed}
The operator $B$ in \eqref{eq:B-maximal}--\eqref{eq:B-domain} is densely
defined and closed.  Its dual operator
\[
 B':D(B')\subset X^*\longrightarrow Y
\]
is densely defined and closed.  If $\phi=(\phi_1,\phi_2)\in D(B')$, then
\begin{equation}\label{eq:B-prime-expression}
 B'\phi=\nabla\phi_1+\nabla\phi_2+a_0\phi_2
\end{equation}
in distributions.  More precisely, its domain is the exact abstract dual
domain
\begin{equation}\label{eq:B-prime-domain}
 D(B')=\left\{\phi\in X^*:
 \begin{array}{c}
 \text{there is a }v\in Y\text{ such that}\\[-2pt]
 \langle Bq,\phi\rangle_{X,X^*}
 =\langle q,v\rangle_{Y^*,Y}\quad(q\in D(B))
 \end{array}\right\},
\end{equation}
and $B'\phi=v$.  Thus \eqref{eq:B-prime-expression} is a necessary
distributional identity, while the boundary condition is encoded by
\eqref{eq:B-prime-domain}; the formal expression alone does not characterize
the domain.  In particular, $B''=B$.
\end{lemma}

\begin{proof}
Every vector field in $C_c^\infty(B_R;\R^3)$ belongs to $D(B)$.

To verify density in the singularly weighted space, let $q\in Y^*$ and
choose radial cutoffs $\chi_\eps\in C_c^\infty(B_R)$ with
$0\leq\chi_\eps\leq1$ and $\chi_\eps=1$ on $B_{R-2\eps}$.  Absolute
continuity of the weighted integral gives
\[
 \|(1-\chi_\eps)q\|_{Y^*}^2
 \leq\int_{B_R\setminus B_{R-2\eps}}
       \frac{|q|^2}{\rho_0}\ud x\longrightarrow0.
\]
For fixed $\eps$, the weight $1/\rho_0$ is bounded above and below on a
compact neighborhood of $\operatorname{supp}\chi_\eps$.  Ordinary
mollification therefore approximates $\chi_\eps q$ in the $Y^*$ norm by
fields in $C_c^\infty(B_R;\R^3)$.  Hence this test class is dense and $B$
is densely defined.

Suppose that $q_n\to q$ in $Y^*$ and $Bq_n\to(g_1,g_2)$ in $X$.  On each
compact set $K\Subset B_R$, weighted Cauchy--Schwarz gives
\[
 \|q_n-q\|_{L^1(K)}
 \leq \|q_n-q\|_{Y^*}\left(\int_K\rho_0\ud x\right)^{1/2}\longrightarrow0.
\]
For each $i$, weighted Cauchy--Schwarz and
$w_i^{-1}\in L^1(K)$ give
\[
 \|B_iq_n-g_i\|_{L^1(K)}
 \leq \|B_iq_n-g_i\|_{X_i}
       \left(\int_Kw_i^{-1}\ud x\right)^{1/2}\longrightarrow0.
\]
Here the assertion for $w_1$ follows from interior positivity, while that for
$w_2$ is part of $\mathrm{(E_s)}$ and allows $\mathfrak G$ to vanish at the
center.  The coefficient $a_0$ is bounded on $K$.
Passing to the limit in distributions yields
\[
 -\Div q=g_1,
 \qquad
 -\Div q+a_0\cdot q=g_2.
\]
Hence $q\in D(B)$ and $Bq=(g_1,g_2)$, proving closedness.

The general duality theorem for a densely defined closed operator now shows
that $B'$ is densely defined and closed and that $B''=B$.
Definition \eqref{eq:B-prime-domain} is precisely the Hilbert-space dual
definition and also proves uniqueness of $v$.

Before testing distributionally, note that every term is locally integrable.
If $K\Subset B_R$, then
\[
 \|\phi_i\|_{L^1(K)}
 \leq \|\phi_i\|_{X_i^*}
       \left(\int_Kw_i\ud x\right)^{1/2}<\infty,
 \qquad i=1,2.
\]
Here $w_1$ and $w_2=\mathfrak G$ are locally bounded on $K$.  Also
$a_0\in L^\infty(K)$, so
$a_0\phi_2\in L^1(K)$, while $v=B'\phi\in Y$ belongs to $L^1(K)$ because
$\rho_0$ is bounded below on $K$.

Finally, testing the defining identity
\[
 \langle Bq,\phi\rangle_{X,X^*}
 =\langle q,B'\phi\rangle_{Y^*,Y}
\]
with $q\in C_c^\infty(B_R;\R^3)$ and integrating by parts gives
\eqref{eq:B-prime-expression}.
\end{proof}

\begin{remark}\label{rem:normalization-LZ}
The maximal-graph proof avoids an unnecessary boundary-trace argument.  If
one instead conjugates the two components to unweighted $L^2$ spaces, the
correct operators are
\[
 -\sqrt{w_1}\,\Div(\sqrt{\rho_0}\,\cdot),
 \qquad
 \sqrt{w_2}\bigl[-\Div(\sqrt{\rho_0}\,\cdot)
                 +a_0\cdot\sqrt{\rho_0}\,\cdot\bigr].
\]
Thus the codomain weights multiply the differential expressions; they do not
occur with inverse square roots.  This is the normalization used for the
barotropic operator in \cite{LinZeng2022}.
\end{remark}

\subsection{Compactness of the potential block}

Let $I_X:X^*\to X$ denote the Riesz isomorphism.  The Riesz representative of
$\mathbb L$ is
\begin{equation}\label{eq:Riesz-L}
 I_X\mathbb L
 =\begin{pmatrix}
 \Id-\dfrac{4\pi}{w_1}(-\Delta)^{-1}&0\\
 0&\Id
 \end{pmatrix}.
\end{equation}

The compact-potential argument of Lemma 3.6 in \cite{LinZeng2022} applies to the
effective pressure $\widetilde P$.
\begin{lemma}\label{lem:potential-compact}
The operator
\[
 \eta\longmapsto \frac1{w_1}(-\Delta)^{-1}\eta
 \quad\hbox{on }X_1
\]
is compact and self-adjoint.  Consequently, $\mathbb L$ is a bounded
self-dual Fredholm operator with finite Morse index and finite-dimensional
kernel.
\end{lemma}
\subsection{The Hamiltonian group and its invariant splitting}

Set
\begin{equation}\label{eq:bold-L-J}
 \mathbf L=\begin{pmatrix}\mathbb L&0\\0&A\end{pmatrix},
 \qquad
 \mathbf J=\begin{pmatrix}0&B\\-B'&0\end{pmatrix}.
\end{equation}
The natural generator domain is
\begin{equation}\label{eq:generator-domain}
 D(\mathbf J\mathbf L)
 =\left\{(z,u)\in X\times Y:
 Au\in D(B),\ \mathbb Lz\in D(B')\right\}.
\end{equation}
Equations \eqref{eq:z-kinematic-equation} and
\eqref{eq:linearized-momentum-separated} are precisely
\begin{equation}\label{eq:separable-Hamiltonian}
 \partial_t\binom zu
 =\mathbf J\mathbf L\binom zu.
\end{equation}

\begin{corollary}\label{thm:Hamiltonian-index}
The operator $\mathbf J\mathbf L$ in \eqref{eq:bold-L-J} generates a
strongly continuous group on $X\times Y$.  Its unstable spectral subspace is
finite-dimensional, all its nonzero unstable eigenvalues are real, and
\begin{equation}\label{eq:global-index-formula}
 n^u
 =n^-\!\left(\mathbb L\big|_{\overline{\Ran B}}\right).
\end{equation}
More precisely, the group has an invariant exponential trichotomy
\[
 X\times Y=E^u\oplus E^c\oplus E^s,
\]
where the spectrum on $E^u$ consists of positive real eigenvalues, the
spectrum on $E^s$ consists of the corresponding negative real eigenvalues,
and
\[
 \dim E^u=\dim E^s=n^u
\]
with algebraic multiplicity.  If the operator induced by the constrained
form on $\overline{\Ran B}$ is an isomorphism onto the dual of that space
(uniform nondegeneracy), then the center group is uniformly bounded; without
that stronger property the abstract theorem gives
\begin{equation}\label{eq:center-quadratic-bound}
 \left\|e^{t\mathbf J\mathbf L}\big|_{E^c}\right\|
 \leq C(1+|t|)^2,
 \qquad t\in\R.
\end{equation}
\end{corollary}

\begin{proof}
  By \cref{lem:B-closed}, the pair $B,B'$ satisfies the closedness and density
hypothesis of the separable Hamiltonian theorem.  The operator $A$ is positive
definite, while \cref{lem:potential-compact} supplies the finite-index
decomposition for $\mathbb L$.  We check the dual-domain condition explicitly.
Let $N=\Ker\mathbb L$ and choose a basis $n_1,\ldots,n_k$ of $N$.
Density of $D(B')$ in $X^*$ permits the choice of
$\phi_1,\ldots,\phi_k\in D(B')$ such that
$(\langle\phi_i,n_j\rangle)_{i,j}$ is invertible.  Then
\[
 X_0=\bigcap_{i=1}^k\Ker\phi_i,\qquad X=N\oplus X_0,
 \qquad X_0^\circ=\Span\{\phi_1,\ldots,\phi_k\}\subset D(B').
\]
The restriction of $\mathbb L$ to $X_0$ is nondegenerate and has the
required positive/negative decomposition.  This verifies the domain
condition associated with its kernel; the corresponding condition for
$A$ is vacuous because $\Ker A=\{0\}$.  Here $X_0^\circ$ denotes the annihilator in $X^*$.  This is the finite-kernel argument
in \cite[Remark~2.2]{LinZeng2022}.
The separable Hamiltonian index and trichotomy theorem
\cite[Theorem~2.1]{LinZeng2022} now applies.  Its refinement for injective
$A$ gives the exponent two in \eqref{eq:center-quadratic-bound}, and its
uniform-nondegeneracy refinement gives the stated uniform bound
\cite[Theorem~2.2(i),(iii)]{LinZeng2022}.

For completeness, the index in that theorem is initially written on
$\overline{\Ran(BA)}$.  Since the Riesz map $A:Y\to Y^*$ is onto, every
$q\in D(B)$ has the form $q=Au$ with $u=A^{-1}q$ in the domain of $BA$.
Consequently
\[
 \Ran(BA)=B(D(B))=\Ran B,
\]
which gives precisely \eqref{eq:global-index-formula}.
\end{proof}

The following invariant splitting isolates velocity perturbations that do not
change either density or entropy.

\begin{lemma}\label{lem:invariant-velocity-splitting}
Let
\[
 Y_1=\Ker(BA),
 \qquad
 Y_2=\overline{\Ran B'}.
\]
Then
\begin{equation}\label{eq:Y-splitting}
 Y=Y_1\oplus^{\perp_Y}Y_2,
\end{equation}
and both $\{0\}\times Y_1$ and $X\times Y_2$ are invariant under the group in
\cref{thm:Hamiltonian-index}.  The group is the identity on
$\{0\}\times Y_1$.
\end{lemma}

\begin{proof}
Use the Hilbert inner product $(u,v)_Y=\langle Au,v\rangle$.  For
$u\in\Ker(BA)$ and $B'\phi\in\Ran B'$,
\[
 (u,B'\phi)_Y=\langle Au,B'\phi\rangle
 =\langle BAu,\phi\rangle=0.
\]
Conversely, the standard closed-operator identity gives
\[
 (\Ran B')^\perp=\Ker(BA),
\]
and proves \eqref{eq:Y-splitting}.

Let $P_1$ be the $Y$-orthogonal projection onto $Y_1$ and define
$\mathbf P(z,u)=(0,P_1u)$.  The projection preserves
$D(\mathbf J\mathbf L)$ in \eqref{eq:generator-domain}: if
$Au\in D(B)$, then $AP_1u\in D(B)$ and
$BAP_1u=0$.  Moreover, $P_1B'\mathbb Lz=0$.  Hence
\[
 \mathbf P\mathbf J\mathbf L
 =\mathbf J\mathbf L\mathbf P=0
 \quad\hbox{on }D(\mathbf J\mathbf L).
\]
Thus $\mathbf P$ commutes with a resolvent and therefore with the generated
group.  Its range and kernel are the two asserted invariant subspaces.
\end{proof}

\section{The non-radial sector}
\label{sec:nonradial}

Throughout this section we assume $\mathrm{(P)}$ and $\mathrm{(E_s)}$.
We first define the non-radial phase space.  This is needed because the
velocity contains a stationary part which is not represented by the scalar
spherical-harmonic decomposition.

For $Q\in SO(3)$ let
\[
 (\mathcal R_Qz)(x)=z(Q^{-1}x),\qquad
 (\mathcal R_Qu)(x)=Q u(Q^{-1}x).
\]
These actions are unitary on $X$ and $Y$, preserve the generator domain, and
commute with $A$, $B$, $B'$, and $\mathbb L$.  Let
\[
 P_{\mathrm{rad}}=\int_{SO(3)}\mathcal R_Q\ud Q
\]
be the orthogonal projection obtained from normalized Haar measure.  Set
\[
\begin{aligned}
 X_{\mathrm{rad}}&=P_{\mathrm{rad}}X,
 &X_{\mathrm{nr}}&=(I-P_{\mathrm{rad}})X,\\
 Y_{2,\mathrm{rad}}&=P_{\mathrm{rad}}Y_2,
 &Y_{2,\mathrm{nr}}&=(I-P_{\mathrm{rad}})Y_2.
\end{aligned}
\]
The elements of $X_{\mathrm{rad}}$ are pairs of radial scalar functions and
the elements of $Y_{2,\mathrm{rad}}$ are radial vector fields.  Moreover,
$Y_1\cap P_{\mathrm{rad}}Y=\{0\}$.  Indeed, if
$u=q(r)e_r\in Y_1$, where $e_r=x/|x|$ is the radial unit
vector, then the two equations in $BAu=0$ imply
$\rho_0'(r)q(r)=0$; hence $u=0$ because $\rho_0'<0$ on $(0,R)$.

The full phase space therefore has the orthogonal invariant decomposition
\begin{equation}\label{eq:radial-nonradial-phase-decomposition}
 X\times Y
 =\underbrace{(X_{\mathrm{rad}}\times Y_{2,\mathrm{rad}})}_{
      \mathscr H_{\mathrm{rad}}}
 \oplus
 \underbrace{\bigl[(\{0\}\times Y_1)
       \oplus(X_{\mathrm{nr}}\times Y_{2,\mathrm{nr}})\bigr]}_{
      \mathscr H_{\mathrm{nr}}}.
\end{equation}
We call the elements of $\mathscr H_{\mathrm{nr}}$ the
\emph{non-radial perturbations}.  The first summand
$\{0\}\times Y_1$ consists of stationary motions tangent to the density
spheres.  The active non-radial variables belong to
$X_{\mathrm{nr}}\times Y_{2,\mathrm{nr}}$ and satisfy the restricted
separable Hamiltonian system
\begin{equation}\label{eq:nonradial-separable-system}
 \partial_t\binom{z_{\mathrm{nr}}}{u_{\mathrm{nr}}}
 =\begin{pmatrix}0&B_{\mathrm{nr}}\\-B_{\mathrm{nr}}'&0\end{pmatrix}
  \begin{pmatrix}\mathbb L_{\mathrm{nr}}&0\\0&A_{\mathrm{nr}}\end{pmatrix}
  \binom{z_{\mathrm{nr}}}{u_{\mathrm{nr}}},
\end{equation}
where the operators are the restrictions of $B$, $\mathbb L$, and $A$ to
the indicated invariant subspaces, with
\[
 D(B_{\mathrm{nr}})=D(B)\cap A(Y_{2,\mathrm{nr}}).
\]
The dual $B_{\mathrm{nr}}'$ is taken with respect to these closed subspaces,
so \eqref{eq:nonradial-separable-system} has the same maximal-domain
interpretation as \eqref{eq:separable-Hamiltonian}.

For scalar functions the decomposition is the usual real
spherical-harmonic expansion in the angular variable
$\omega=x/|x|\in\mathbb S^2$
\[
 h(r,\omega)=\sum_{\ell=0}^\infty\sum_{m=-\ell}^{\ell}
 h_{\ell m}(r)Y_{\ell m}(\omega),
 \qquad
 X_{\mathrm{nr}}
 =\overline{\bigoplus_{\ell\geq1}X^{(\ell)}}.
\]
The vector decomposition is the corresponding vector-spherical-harmonic
decomposition.  Its toroidal component lies in $Y_1$ and is stationary; the
poloidal component belongs to the active space $Y_2$.  This identifies
precisely the invariant meaning of ``non-radial'' used below.

\subsection{Density--potential inertia}

Define the self-adjoint Schr\"odinger form on $\dot H^1(\R^3)$
\begin{equation}\label{eq:D0-form}
 \langle D_0\varphi,\varphi\rangle
 =\int_{\R^3}|\nabla\varphi|^2\ud x
 -4\pi\int_{B_R}\frac{|\varphi|^2}{w_1}\ud x,
 \qquad
 D_0=-\Delta-\frac{4\pi}{w_1}\mathbf1_{B_R}.
\end{equation}

The next two lemmas use the density--potential reduction of Lemma 3.7-3.8 in
\cite{LinZeng2022}, applied to the effective pressure $\widetilde P$.
\begin{lemma}\label{lem:density-potential-inertia}
For every angular sector $\ell\geq0$,
\begin{equation}\label{eq:L11-D0-inertia}
 n^-\!\left(L_{11}\big|_{X_1^{(\ell)}}\right)
 =n^-\!\left(D_0\big|_{\dot H^1_{(\ell)}}\right),
 \qquad
 \dim\Ker L_{11}^{(\ell)}=\dim\Ker D_0^{(\ell)}.
\end{equation}
\end{lemma}

Writing $\Delta_r=\partial_r^2+2r^{-1}\partial_r$, the radial representative
of \eqref{eq:D0-form} in the $\ell$-sector is
\begin{equation}
 D_0^{(\ell)}
 =-\Delta_r+\frac{\ell(\ell+1)}{r^2}
  -\frac{4\pi}{w_1}\mathbf1_{(0,R)}.
\end{equation}

\begin{lemma}\label{lem:D0-nonradial-positive}
One has
\begin{enumerate}[label=(\roman*)]
\item $D_0^{(1)}\geq0$ and
$\Ker D_0^{(1)}=\Span\{V_0'\}$;
\item the quadratic form of $D_0^{(\ell)}$ is positive definite for every
$\ell\geq2$.
\end{enumerate}
Consequently $L_{11}^{(\ell)}\geq0$ for all $\ell\geq1$.  Its only
non-radial kernel is the three-dimensional translational kernel in the
$\ell=1$ sector.
\end{lemma}
\begin{theorem}[Non-radial stability]\label{thm:nonradial-stability}
The Schwarzschild-stable equilibrium is mode stable in the non-radial sector,
and
\begin{equation}\label{eq:nonradial-index-zero}
 n^u_{\mathrm{nr}}=0.
\end{equation}
The $\ell=1$ translations are neutral modes.  

\end{theorem}

\begin{proof}
For $z=(\eta,\zeta)\in X^{(\ell)}$ with $\ell\geq1$,
\[
 \langle\mathbb Lz,z\rangle
 =\langle L_{11}\eta,\eta\rangle
   +\int_{B_R}\mathfrak G|\zeta|^2\ud x\geq0
\]
by \eqref{eq:L11-D0-inertia},
\cref{lem:D0-nonradial-positive} and $\mathfrak G>0$.  Hence the
restriction of $\mathbb L$ to the non-radial part of
$\overline{\Ran B}$ has zero negative index.  Formula
\eqref{eq:global-index-formula} gives \eqref{eq:nonradial-index-zero}.
The kernel statement follows from the last part of
\cref{lem:D0-nonradial-positive}.
\end{proof}

\begin{remark}\label{rem:Lebovitz-comparison}
Building on the variational treatment of non-radial oscillations in
\cite{Chandrasekhar1964,ChandrasekharLebovitz1964}, Lebovitz
\cite{Lebovitz1965} separated the displacement equations into
spherical harmonics and proved positivity of the characteristic-frequency
variational expression for every $\ell>1$ under the Schwarzschild sign.  The
kernel in the $\ell=1$ sector consists of uniform translations; this
does not identify the whole dipole sector with that kernel.  The argument above has
the same angular conclusion, but it is formulated on the invariant phase
space \eqref{eq:radial-nonradial-phase-decomposition}: the stationary
toroidal space is separated first, and the Hamiltonian index formula rules
out growing modes without assuming completeness of a normal-mode expansion.
Lebovitz's companion note \cite{LebovitzOnset1965} likewise separates the
trivial rotational motions before characterizing neutral convective modes
when the discriminant vanishes on a layer.  His necessity argument
\cite{Lebovitz1966} constructs a destabilizing displacement when the
discriminant is negative; the closed velocity-form argument in
\cref{sec:convective} is the corresponding energy-space construction.  The
present theorem does not rule out neutral or oscillatory accumulation at
zero.
\end{remark}

\begin{remark}\label{rem:g-p-modes}
The theorem concerns exponentially growing modes.  The finer oscillatory
spectrum is a separate question.  Under additional hypotheses, buoyancy
($g$-) modes can accumulate at zero and acoustic ($p$-) modes can have
unbounded frequencies; see \cite{Makino2023} and the astrophysical
references cited in the introduction.  No such discrete-mode classification
is asserted here under $\mathrm{(P)}$ and $\mathrm{(E_s)}$ alone.
\end{remark}

\section{Radial reduction and the constrained quadratic form}
\label{sec:radial}

Throughout this section we assume $\mathrm{(P)}$ and $\mathrm{(E_s)}$.

We now restrict to spherically symmetric density and entropy perturbations and
radial velocity fields.  A superscript ``rad'' denotes the corresponding
closed subspace.  The main point is that the infinitely many entropy-Casimir
constraints reduce to a single mass constraint plus an explicit entropy
reconstruction.

\subsection{The energy--Casimir Hessian}

The conserved total energy is
\begin{equation}\label{eq:total-energy}
 \cE(\rho,S,u)
 =\int_{\R^3}\left(\frac12\rho|u|^2+\Phi(S,\rho)\right)\ud x
  -\frac1{8\pi}\int_{\R^3}|\nabla V_\rho|^2\ud x.
\end{equation}
For $g\in C^1(\R)$, the entropy-Casimir
\begin{equation}\label{eq:entropy-Casimir}
 \mathcal C_g(\rho,S)=\int_{\R^3}\rho g(S)\ud x
\end{equation}
is conserved by \eqref{eq:EP}.

Since $P_S>0$, $f'<0$, and $\rho_0'<0$, the equilibrium entropy $S_0$ is
strictly increasing on $(0,R)$.  Put
$J_0=S_0((0,R))$ and define $g_0\in C^2_{\mathrm{loc}}(J_0)$, up to an
additive constant, by
\begin{equation}\label{eq:g0-definition}
 g_0'(S_0(r))=-\frac{\Phi_S(S_0(r),\rho_0(r))}{\rho_0(r)}.
\end{equation}
The pressure asymptotics imply that $g_0'$ has a finite limit at the vacuum
endpoint (in fact that limit is zero), so $g_0$ has a $C^1$ extension to the
closed entropy range and then to $\R$.  We do not assume that $g_0''$ extends
to the vacuum endpoint.  This distinction is needed, in particular, for
entropy profiles prescribed in Lagrangian mass coordinates.  On every
interior ball $B_{R-\delta}$, \eqref{eq:pressure-gradient-identity} and the
hydrostatic equation give a constant $c_0$, independent of $\delta$, such
that
\begin{equation}\label{eq:energy-Casimir-Euler-Lagrange}
 \Phi_\rho(S_0,\rho_0)+V_0+g_0(S_0)+c_0=0,
 \qquad
 \Phi_S(S_0,\rho_0)+\rho_0g_0'(S_0)=0.
\end{equation}

\Needspace{12\baselineskip}
\begin{lemma}\label{lem:energy-casimir-hessian}
The first variation of
\[
 \cH_0=\cE+c_0\int_{\R^3}\rho\ud x+\mathcal C_{g_0}
\]
vanishes at the equilibrium for smooth perturbations supported in a
compact annulus $0<\eps<|x|<R-\delta$.  For such a perturbation
$(\eta,s,u)$ its second variation is
\begin{equation}\label{eq:energy-Casimir-Hessian}
\begin{split}
 \delta^2\cH_0[(\eta,s,u)]
={}&\left\langle
 \bigl(w_1-4\pi(-\Delta)^{-1}\bigr)\eta,\eta
 \right\rangle
 +\int_{B_R}\mathfrak G
       \left|\eta-\frac{s}{f'(\rho_0)}\right|^2\ud x\\
 &+\int_{B_R}\rho_0|u|^2\ud x.
\end{split}
\end{equation}
Equivalently, in the variables \eqref{eq:eta-zeta-variables}, this is
$\langle\mathbf L(z,u),(z,u)\rangle$.  The right-hand side extends uniquely
to a bounded quadratic form on $X\times Y$.
\end{lemma}

\begin{proof}
The first variation follows from
\eqref{eq:energy-Casimir-Euler-Lagrange}.  To justify the second variation,
fix $0<\eps<R-\delta$ and first work on the annulus
\[
 A_{\eps,\delta}=\{x:\eps<|x|<R-\delta\}.
\]
The entropy values on this annulus form a compact subset of $J_0$, so
$g_0$ is $C^2$ there.  From
\eqref{eq:thermodynamic-identities} and \eqref{eq:g0-definition}, along the
equilibrium,
\begin{equation}\label{eq:Hessian-coefficient-identities}
 \Phi_{S\rho}+g_0'=\frac{P_S}{\rho_0},
 \qquad
 \Phi_{SS}+\rho_0g_0''=-\frac{P_S}{\rho_0f'(\rho_0)},
 \qquad
 \Phi_{\rho\rho}=w_1+\mathfrak G.
\end{equation}
Indeed, the first identity follows from
$\partial_\rho(\Phi_S/\rho)=P_S/\rho^2$; differentiating
\eqref{eq:g0-definition} along $S_0=f(\rho_0)$ gives the second; and the last
uses
\[
 \frac{P_\rho}{\rho_0}
 =\frac{P_\rho+P_Sf'}{\rho_0}-\frac{P_Sf'}{\rho_0}
 =w_1+\mathfrak G.
\]
Thus, by \eqref{eq:Hessian-coefficient-identities}, the local
density--entropy part of the Hessian is
\[
 (w_1+\mathfrak G)\eta^2
 +2\frac{P_S}{\rho_0}\eta s
 -\frac{P_S}{\rho_0f'}s^2
 =w_1\eta^2+\mathfrak G\left|\eta-\frac{s}{f'}\right|^2.
\]
The gravitational and kinetic variations give the other two terms in
\eqref{eq:energy-Casimir-Hessian}.  Now let $\eps\downarrow0$ and
$\delta\downarrow0$.  We pass to the limit in the completed-square
expression, not in its three uncompleted coefficients.  The phase-space
norms give
\[
 \int w_1|\eta|^2+\int\mathfrak G
 \left|\eta-\frac{s}{f'(\rho_0)}\right|^2
 +\int\rho_0|u|^2<\infty,
\]
so absolute continuity of the integrals makes both annular cutoff errors tend
to zero.  The Newton term converges by its boundedness on $X_1$.  This proves
\eqref{eq:energy-Casimir-Hessian} and its unique bounded extension.  In
particular, no value of $g_0''$ is required at either the central or vacuum
entropy endpoint.
\end{proof}

\subsection{A weighted Hardy estimate and the radial range}

For radial $\eta$, define its enclosed perturbed mass and the reconstruction
operator by
\begin{equation}\label{eq:Ieta-and-T}
 I_\eta(r)=4\pi\int_0^rs^2\eta(s)\ud s,
 \qquad
 (T\eta)(r)=\frac{\rho_0'(r)}{4\pi r^2\rho_0(r)}I_\eta(r).
\end{equation}
The mass functional is continuous on $X_1^{\mathrm{rad}}$, because
$w_1^{-1}\in L^1(B_R)$.  Hence
\begin{equation}\label{eq:mass-zero-space}
 \mathsf H_0
 =\left\{\eta\in X_1^{\mathrm{rad}}:I_\eta(R)=0\right\}
\end{equation}
is a closed codimension-one subspace.

\begin{lemma}\label{lem:radial-Hardy}
The map $T$ extends uniquely to a
bounded operator
\[
 T:\mathsf H_0\longrightarrow X_2^{\mathrm{rad}}.
\]
Moreover, the radial momentum field
\begin{equation}\label{eq:q-eta}
 q_\eta(r)=-\frac{I_\eta(r)}{4\pi r^2}e_r
\end{equation}
belongs to $Y^*$ and
\begin{equation}\label{eq:Hardy-estimates}
 \|q_\eta\|_{Y^*}+\|T\eta\|_{X_2}
 \leq C\|\eta\|_{X_1},
 \qquad \eta\in\mathsf H_0.
\end{equation}
In addition, the inclusion $X_1\hookrightarrow X_2$ is continuous.
\end{lemma}

\begin{proof}
Put
\[
 \alpha=\frac1{\gamma_0-1}>1,
 \qquad
 \beta=\frac{2-\gamma_0}{\gamma_0-1}=\alpha-1.
\]
By \cref{lem:physical-vacuum} and
\eqref{eq:weight-comparability}, with $d=R-r$,
\begin{equation}\label{eq:boundary-weight-summary}
 \rho_0\asymp d^\alpha,
 \quad w_1\asymp d^{-\beta},
 \quad w_2\leq Cd,
 \quad \left|\frac{\rho_0'}{\rho_0}\right|\leq \frac Cd
\end{equation}
near $R$.  Since $I_\eta(R)=0$, weighted Cauchy--Schwarz gives
\begin{equation}\label{eq:Ieta-boundary-estimate}
\begin{split}
 |I_\eta(r)|^2
 &=(4\pi)^2\left|\int_r^Rs^2\eta(s)\ud s\right|^2\\
 &\leq C\left(\int_r^Rw_1|\eta|^2s^2\ud s\right)
             \left(\int_r^R\frac{s^2}{w_1}\ud s\right)
 \leq Cd^\alpha\int_r^Rw_1|\eta|^2s^2\ud s.
\end{split}
\end{equation}

Set
\[
 E_R(r)=\int_r^Rw_1(s)|\eta(s)|^2s^2\ud s.
\]
Then \eqref{eq:Ieta-boundary-estimate} and Tonelli's theorem give the two
endpoint integrations explicitly:
\begin{align*}
 \int_{R-\eps}^R\frac{|I_\eta(r)|^2}{r^2\rho_0(r)}\ud r
 &\leq C\int_{R-\eps}^RE_R(r)\ud r\\
 &=C\int_{R-\eps}^Rw_1(s)|\eta(s)|^2s^2
       \int_{R-\eps}^{s}\ud r\,\ud s
 \leq C\eps\|\eta\|_{X_1}^2,
\end{align*}
and, using \eqref{eq:boundary-weight-summary},
\begin{align*}
 &\int_{R-\eps}^R
 w_2\left|\frac{\rho_0'I_\eta}{r^2\rho_0}\right|^2r^2\ud r
 \leq C\int_{R-\eps}^R\frac{|I_\eta(r)|^2}{R-r}\ud r\\
 &\qquad\leq C\int_{R-\eps}^R(R-r)^{\alpha-1}E_R(r)\ud r\\
 &\qquad=C\int_{R-\eps}^Rw_1(s)|\eta(s)|^2s^2
       \int_{R-\eps}^{s}(R-r)^{\alpha-1}\ud r\,\ud s
 \leq C\eps^\alpha\|\eta\|_{X_1}^2.
\end{align*}
Here $\alpha>1$; in fact, only $\alpha>0$ is needed for these two
integrations.

Near $r=0$, radial $C^2$ regularity gives $\rho_0'(r)=O(r)$, while
\[
 |I_\eta(r)|^2
 \leq Cr^3\int_0^rw_1|\eta|^2s^2\ud s.
\]

Writing $E_0(r)=\int_0^rw_1|\eta|^2s^2\ud s$ and using
$\rho_0\asymp1$ near the center, Tonelli gives
\[
 \int_0^\eps\frac{|I_\eta(r)|^2}{r^2\rho_0(r)}\ud r
 \leq C\int_0^\eps rE_0(r)\ud r
 \leq C\eps^2\|\eta\|_{X_1}^2.
\]
Moreover, $\rho_0'/\rho_0=O(r)$, and therefore
\[
 \int_0^\eps
 w_2\left|\frac{\rho_0'I_\eta}{r^2\rho_0}\right|^2r^2\ud r
 \leq C\int_0^\eps w_2(r)r^3E_0(r)\ud r
 \leq C\eps^4\|\eta\|_{X_1}^2,
\]
because $w_2$ is bounded near the center by the endpoint $C^1$ regularity
of $f$.  On a compact annulus away from $0$
and $R$, all coefficients are bounded and the mass functional
$I_\eta(r)$ is bounded by weighted Cauchy--Schwarz.  This proves
\eqref{eq:Hardy-estimates}.

Finally, $w_2/w_1\leq C\rho_0$ near the boundary and is bounded in the
interior, so $\|\eta\|_{X_2}\leq C\|\eta\|_{X_1}$.
\end{proof}

Set
\begin{equation}\label{eq:K-radial}
 K=\Id-T:\mathsf H_0\longrightarrow X_2^{\mathrm{rad}}.
\end{equation}

\begin{proposition}
\label{prop:radial-range}
One has
\begin{equation}\label{eq:radial-range-graph}
 \Ran B\cap X^{\mathrm{rad}}
 =\overline{\Ran B}\cap X^{\mathrm{rad}}
 =\left\{\binom{\eta}{K\eta}:\eta\in\mathsf H_0\right\}.
\end{equation}
In particular, the radial range is closed.
\end{proposition}

\begin{proof}
First take a smooth radial $\eta\in\mathsf H_0$.  The momentum
$q_\eta$ in \eqref{eq:q-eta} is regular at the center and belongs to $Y^*$
by \cref{lem:radial-Hardy}.  Direct calculation gives
\[
 -\Div q_\eta=\eta,
 \qquad
 -\Div q_\eta+a_0\cdot q_\eta=\eta-T\eta=K\eta.
\]
Thus $Bq_\eta=(\eta,K\eta)$.  Smooth mass-zero functions are dense in
$\mathsf H_0$ defined in \eqref{eq:mass-zero-space}: truncate
and mollify, then correct the mass with one fixed
compactly supported smooth function.  The estimates in
\cref{lem:radial-Hardy} and closedness of $B$ extend the identity to all
$\eta\in\mathsf H_0$.

Conversely, suppose that $Bq$ is radial.  Averaging $q$ over the rotation
group preserves $D(B)$ and commutes with $B$, so we may replace $q$ by its
radial average and write $q=q_r(r)e_r$.  Set
$\eta=-r^{-2}(r^2q_r)'$.  Using a boundary cutoff $\chi_\eps$ with
$|\nabla\chi_\eps|\leq C/\eps$ and
\cref{lem:physical-vacuum},
\[
 \left|\int q\cdot\nabla\chi_\eps\ud x\right|
 \leq
 \|q\|_{L^2_{1/\rho_0}(\{d<2\eps\})}
 \left(\int_{\{d<2\eps\}}\rho_0|\nabla\chi_\eps|^2\ud x\right)^{1/2}
 \longrightarrow0,
\]
because the second factor is $O(\eps^{(\alpha-1)/2})$.  Hence
$\int_{B_R}\eta\ud x=0$.  Integration of the radial divergence equation
gives
\[
 r^2q_r(r)=-\frac{I_\eta(r)}{4\pi}+C.
\]
The center constant is zero: if $C\neq0$, then $q_r\sim Cr^{-2}$ and
$\int_0^\delta |q_r|^2r^2/\rho_0\ud r=\infty$, contrary to $q\in Y^*$.
Thus $q=q_\eta$, and the second component of $Bq$ is $K\eta$.
This proves the first equality in \eqref{eq:radial-range-graph}.  The right
side is the graph of a bounded operator over the closed space $\mathsf H_0$,
so it is closed in $X^{\mathrm{rad}}$.  Finally, if
$Bq_n\to z\in X^{\mathrm{rad}}$, rotation averaging gives
$B\overline q_n=\overline{Bq_n}\to z$, where every
$B\overline q_n$ belongs to the graph just identified.  Its closedness then
gives $z$ in that graph and proves the second equality.
\end{proof}

\subsection{Entropy-Casimir constraints}

For a density--entropy perturbation $(\eta,s)$, the linearization of
\eqref{eq:entropy-Casimir} is
\begin{equation}\label{eq:linearized-Casimir}
 \delta\mathcal C_g[\eta,s]
 =\int_{B_R}\left[g(S_0)\eta+\rho_0g'(S_0)s\right]\ud x.
\end{equation}

Let $\mathscr G_0$ be the class of $g\in C^1(\R)$ such that
$\operatorname{supp}(g'|_{J_0})\Subset J_0$.  This includes the constants.
For $g\in\mathscr G_0$, \eqref{eq:linearized-Casimir} is a continuous
functional on the phase space: its entropy-dependent coefficient is
supported away from the vacuum boundary.  This test class is already large
enough to separate radial entropy perturbations.  If the integral in
\eqref{eq:linearized-Casimir} is known to be continuous for a larger class of
Casimirs, the equivalent constraint automatically holds for that larger
class as well.

\begin{proposition}\label{prop:Casimir-accessibility}
For radial $(\eta,s)$ with
$(\eta,\eta-s/f'(\rho_0))\in X^{\mathrm{rad}}$, the following are equivalent:
\begin{enumerate}[label=(\roman*)]
\item $(\eta,\eta-s/f'(\rho_0))\in\overline{\Ran B}$;
\item
\begin{equation}\label{eq:radial-entropy-reconstruction}
 I_\eta(R)=0,
 \qquad
 s(r)=f'(\rho_0(r))\frac{\rho_0'(r)}{4\pi r^2\rho_0(r)}I_\eta(r);
\end{equation}
\item $\delta\mathcal C_g[\eta,s]=0$ for every $g\in\mathscr G_0$.
\end{enumerate}
\end{proposition}

\begin{proof}
The equivalence of (i) and (ii) follows from
\cref{prop:radial-range}, since
$\eta-s/f'=K\eta$ with $K$ as in \eqref{eq:K-radial} is
equivalent to $s=f'T\eta$.

Assume (ii).  Since $I_\eta'=4\pi r^2\eta$, radial integration gives
\begin{align*}
 \delta\mathcal C_g[\eta,s]
 &=\int_0^R g(S_0)I_\eta'\ud r
   +\int_0^R g'(S_0)f'(\rho_0)\rho_0'I_\eta\ud r\\
 &=\bigl[g(S_0(r))I_\eta(r)\bigr]_{r=0}^{r=R}=0.
\end{align*}
Thus (iii) holds.

Conversely, take $g\equiv1$ in (iii) to get $I_\eta(R)=0$, and set
$h=s-f'T\eta$.  The reconstructed pair $(\eta,f'T\eta)$ already satisfies
all the constraints, so
\[
 \int_{B_R}\rho_0g'(S_0)h\ud x=0
 \qquad\text{for every }g\in\mathscr G_0.
\]
On any compact subinterval of $(0,R)$, strict monotonicity of $S_0$ lets
$g'(S_0(r))$ range over a dense class of continuous test functions.
Consequently $\rho_0h=0$ there.  Exhausting $(0,R)$ by compact subintervals
gives $h=0$ almost everywhere and proves (ii).
\end{proof}

\subsection{The reduced radial form and instability index}

Define the bounded symmetric form on the mass-zero space $\mathsf H_0$ by
\begin{equation}\label{eq:Q0-reduced}
\begin{split}
 Q_0[\eta]
={}&\left\langle
 \bigl(w_1-4\pi(-\Delta)^{-1}\bigr)\eta,\eta
 \right\rangle
 +\int_{B_R}\mathfrak G|K\eta|^2\ud x,
 \qquad \eta\in\mathsf H_0.
\end{split}
\end{equation}
The restriction to $\mathsf H_0$ is essential: if $I_\eta(R)\neq0$, the
reconstruction term need not belong to $X_2$ at the vacuum boundary.

\begin{lemma}\label{lem:Q0-Fredholm}
The restriction $Q_0|_{\mathsf H_0}$ is a bounded symmetric Fredholm form
with finite Morse index.  If $\mathscr A_0$ denotes its Riesz representative,
then, schematically,
\begin{equation}\label{eq:Q0-identity-compact}
 \mathscr A_0=\Id+K^*K+\mathscr C,
\end{equation}
where $K^*K\geq0$ and $\mathscr C$ is compact.
\end{lemma}

\begin{proof}
The boundedness of $K:\mathsf H_0\to X_2^{\mathrm{rad}}$ follows from
\cref{lem:radial-Hardy}.  Relative to the $X_1$ inner product, multiplication
by $w_1$ is the identity, the Schwarzschild term is $K^*K$, and the Newton
term is compact by \cref{lem:potential-compact}.  Thus
$\Id+K^*K$ is coercive and \eqref{eq:Q0-identity-compact} is Fredholm.  A
finite-codimensional restriction preserves the Fredholm property and finite
Morse index.
\end{proof}

\begin{theorem}\label{thm:radial-index}
The algebraic number of exponentially growing radial modes is
\begin{equation}\label{eq:radial-index-formula}
 n^u_{\mathrm{rad}}
 =n^-\!\left(Q_0\big|_{\mathsf H_0}\right).
\end{equation}
In particular, the equilibrium is radially mode stable if and only if
$Q_0[\eta]\geq0$ for every $\eta\in\mathsf H_0$.
\end{theorem}

\begin{proof}
By \cref{prop:radial-range}, the pullback
\[
 \eta\longmapsto(\eta,K\eta)
\]
identifies $\mathsf H_0$ with the radial part of $\overline{\Ran B}$.  Under
this identification, the restriction of $\mathbb L$ is exactly
\eqref{eq:Q0-reduced}.  Formula \eqref{eq:radial-index-formula} follows from
\eqref{eq:global-index-formula}.
\end{proof}

\begin{corollary}[Full growing-mode count]\label{cor:full-growing-count}
Under $\mathrm{(P)}$ and $\mathrm{(E_s)}$, every exponentially growing mode
is radial and
\[
 \dim E^u=\dim E^s=n^u
 =n^-\!\left(Q_0\big|_{\mathsf H_0}\right).
\]
The center group satisfies \eqref{eq:center-quadratic-bound}.
\end{corollary}

\begin{proof}
Combine \cref{thm:nonradial-stability,thm:radial-index} with the global index
formula \eqref{eq:global-index-formula} and the trichotomy in
\cref{thm:Hamiltonian-index}.
\end{proof}

\subsection{A one-constraint inertia lemma}

We record the abstract calculation used for both parameterized families.
The null case is included because it is precisely the case occurring at a
mass extremum.

\begin{lemma}\label{lem:one-constraint-inertia}
Let $\mathsf H$ be a real Hilbert space and let
$Q:\mathsf H\times\mathsf H\to\R$ be a bounded symmetric Fredholm form with
finite Morse index.  Let $\ell\in\mathsf H^*$ be nonzero and suppose that
there are $v\in\mathsf H$ and $c\in\R\setminus\{0\}$ such that
\begin{equation}\label{eq:constraint-normal-identity}
 Q(v,h)=-c\,\ell(h)
 \qquad\text{for every }h\in\mathsf H.
\end{equation}
Set $\mathsf H_\ell=\Ker\ell$ and
\begin{equation}\label{eq:q-normal}
 q=Q(v,v)=-c\,\ell(v).
\end{equation}
Then $Q|_{\mathsf H_\ell}$ is Fredholm and
\begin{equation}\label{eq:one-constraint-index}
 n^-\!\left(Q\big|_{\mathsf H_\ell}\right)
 =\begin{cases}
 n^-(Q)-1,&q<0,\\
 n^-(Q),&q>0,\\
 n^-(Q)-1,&q=0.
 \end{cases}
\end{equation}
If $N=\Ker Q$ denotes the full-space radical, then
\begin{equation}\label{eq:one-constraint-kernel}
 \Ker\!\left(Q\big|_{\mathsf H_\ell}\right)
 =\begin{cases}
 N,&q\neq0,\\
 N\oplus\Span\{v\},&q=0.
 \end{cases}
\end{equation}
Here the kernel on the left is the radical of the restricted form.
\end{lemma}

\begin{proof}
For $n\in N$, \eqref{eq:constraint-normal-identity} gives $\ell(n)=0$.
Suppose first that $q\neq0$.  Then
$\ell(v)=-q/c\neq0$, and
\[
 \mathsf H=\Span\{v\}\oplus\mathsf H_\ell.
\]
This is a $Q$-orthogonal direct sum because $Q(v,h)=0$ on
$\mathsf H_\ell$.  Sylvester additivity proves the first two cases of
\eqref{eq:one-constraint-index} and the first case of
\eqref{eq:one-constraint-kernel}.

Now let $q=0$.  Then $v\in\mathsf H_\ell$ and
$Q(v,h)=0$ for every $h\in\mathsf H_\ell$.  Choose $z$ with $\ell(z)=1$.
Since $Q(v,z)=-c$, replacing $z$ by
\[
 z+\frac{Q(z,z)}{2c}v
\]
makes $Q(z,z)=0$ without changing $\ell(z)$.  On
$E=\Span\{v,z\}$, the matrix of $Q$ is
\[
 \begin{pmatrix}0&-c\\-c&0\end{pmatrix},
\]
so $E$ is nondegenerate and has one positive and one negative direction.
Therefore
\[
 \mathsf H=E\oplus E^{\perp_Q}.
\]
For $w\in E^{\perp_Q}$,
$0=Q(v,w)=-c\ell(w)$, and hence
\[
 \mathsf H_\ell=\Span\{v\}\oplus E^{\perp_Q}.
\]
The form vanishes on the first summand.  Exactly one negative direction is
lost, and the restricted radical is $N\oplus\Span\{v\}$.  All these decompositions are topological: the projections onto the
finite-dimensional nondegenerate summands are bounded and are obtained
by inverting their Gram matrices.  Removing such a summand from a
Fredholm form leaves a Fredholm form, and adjoining the one-dimensional
zero block in the null case preserves Fredholmness.  This proves the
remaining assertions.
\end{proof}

\begin{remark}\label{rem:constraint-lemma-two-families}
In the fixed entropy--density family, the family tangent is the quadratic
normal to the mass hyperplane for the barotropic block $L_{11}$, but generally
not for the full form $Q_0$ because of the positive term involving $K$.  In
the fixed Lagrangian-entropy family, the variational identity proved later
shows that the tangent is the normal for the full reduced Hessian.  This is
the structural reason the two turning-point conclusions differ.  If the
coefficient $c$ in \eqref{eq:constraint-normal-identity} vanishes, the lemma
does not apply and that parameter must be treated separately.
\end{remark}

\section{A fixed entropy--density relation: existence and failure of the
turning-point principle}\label{sec:fixed-relation}

We fix a single entropy--density relation along the family:
\begin{equation}\label{eq:fixed-relation-family}
 S_\mu(r)=f(\rho_\mu(r)),\qquad \rho_\mu(0)=\mu,
\end{equation}
where $f$ is not assumed small.  Set
\begin{equation}\label{eq:fixed-relation-effective-pressure}
 \widetilde P_f(\rho)=P(f(\rho),\rho),\qquad
 \mathfrak G_f(\rho)=-\frac{P_S(f(\rho),\rho)f'(\rho)}{\rho}.
\end{equation}
The equilibrium equation is therefore the barotropic hydrostatic equation
with pressure $\widetilde P_f$.  The second term in the radial stability form,
however, still records that entropy is advected.  All indices in this and the
next section are radial.

The effective-compressibility assumption still restricts the slope:
\[
 0<-f'(\rho)<\frac{P_\rho(f(\rho),\rho)}{P_S(f(\rho),\rho)}.
\]
Thus absence of an entropy-amplitude smallness assumption does not mean
that every decreasing function $f$ is admissible.

\subsection{Existence for a fixed relation}

\begin{lemma}
\label{lem:fixed-relation-existence}
Assume $\mathrm{(P)}$.  Let
$f\in C^1([0,\bar\mu])\cap C^3((0,\bar\mu])$ satisfy
\begin{equation}\label{eq:fixed-relation-compressibility}
 f'(\rho)<0,\qquad \widetilde P_f'(\rho)>0
 \quad(0<\rho\leq\bar\mu).
\end{equation}
Then there is $\mu_0\in(0,\bar\mu]$ such that, for every
$0<\mu<\mu_0$, \eqref{eq:fixed-relation-family} has a unique radial
equilibrium up to a finite first vacuum radius $R_\mu$.  It is strictly
decreasing, has positive mass $M_\mu$, and has the physical-vacuum behavior
\eqref{eq:physical-vacuum}--\eqref{eq:physical-vacuum-derivative}; in
particular, it satisfies $\mathrm{(E_s)}$.  
\end{lemma}

\begin{proof}
The pressure asymptotics in $\mathrm{(P)}$ and the boundedness of $f'$ near
zero give
\begin{equation}\label{eq:fixed-relation-pressure-asymptotics}
 \widetilde P_f'(\rho)
 =\gamma_0K(f(0))\rho^{\gamma_0-1}(1+o(1))
 \qquad(\rho\downarrow0).
\end{equation}
Indeed, the $P_\rho$ term has the displayed leading part, whereas
$P_S(f(\rho),\rho)f'(\rho)=O(\rho^{\gamma_0})$.
Thus $\widetilde P_f$ satisfies the small-density
barotropic hypotheses \eqref{eq:fixed-relation-compressibility}
of the low-density existence lemma in \cite{LinZeng2022} (Lemma~3.1 of its author version), which gives a finite-radius solution for
every sufficiently small central density.

For clarity, the radial initial-value problem used there can be written
directly.  Define
\begin{equation}\label{eq:fixed-relation-enthalpy}
 H_f(\rho)=\int_0^\rho\frac{\widetilde P_f'(q)}q\ud q,
 \qquad F_f=H_f^{-1}.
\end{equation}
While $y=H_f(\rho)>0$,
\begin{equation}\label{eq:fixed-relation-enthalpy-system}
 m(r)=4\pi\int_0^r s^2F_f(y(s))\ud s,
 \qquad
 y(r)=H_f(\mu)-\int_0^r\frac{m(s)}{s^2}\ud s.
\end{equation}
The system \eqref{eq:fixed-relation-enthalpy-system} is
regular at the center, since $m(r)=O(r^3)$, and its local
solution is unique.  Moreover, $y'=-m/r^2<0$ away from the center, so the
first zero supplied by the cited lemma is the first vacuum and the density is
strictly decreasing.  Finally,
\eqref{eq:fixed-relation-pressure-asymptotics} gives
\[
 H_f(\rho)\asymp\rho^{\gamma_0-1},\qquad
 H_f'(\rho)\asymp\rho^{\gamma_0-2}.
\]
At the first zero $R_\mu$,
$-y'(r)=m(r)/r^2\to M_\mu/R_\mu^2>0$.  Hence
$y(r)\asymp R_\mu-r$, and inversion of $H_f$ gives
\eqref{eq:physical-vacuum}.  The identity
$\rho'=y'/H_f'(\rho)$ with $H_f$ from
\eqref{eq:fixed-relation-enthalpy} then gives
\eqref{eq:physical-vacuum-derivative} and the boundary regularity in
\eqref{eq:equilibrium-density-regularity}.  Moreover,
$P_S(f(\rho),\rho)=O(\rho^{\gamma_0})$ and boundedness of $f'$ give
$0<\mathfrak G_f\leq C\rho^{\gamma_0-1}$; on compact interior density
intervals its reciprocal is integrable.  Thus all conditions in
$\mathrm{(E_s)}$ hold.
\end{proof}

The preceding lemma supplies a nonempty branch without any entropy-amplitude
restriction.  A mass maximum need not lie in its small-central-density part.
Accordingly, for the turning-point statement we take any compact continuation
$I=[\mu_-,\mu_+]$ on which the first vacuum remains finite and the family is
$C^2$ and satisfies $\mathrm{(E_s)}$ and $\mathrm{(F)}$.  These are explicit
branch hypotheses, not smallness assumptions.  The fixed-endpoint device used
to verify parameter regularity is still available: after the first crossing
of a fixed enthalpy level $y_0>0$, use $y$ as independent variable and solve
\begin{equation}\label{eq:fixed-relation-fixed-enthalpy-system}
 \frac{\ud r}{\ud y}=-\frac{r^2}{m},\qquad
 \frac{\ud m}{\ud y}=-\frac{4\pi r^4}{m}F_f(y),
 \qquad 0\leq y\leq y_0.
\end{equation}
Because $r$ and $m$ are bounded away from zero on this surface interval,
ordinary parameter dependence for
\eqref{eq:fixed-relation-fixed-enthalpy-system} is taken on a fixed
domain; no differentiation
of a zero extension across the moving vacuum is involved.

\subsection{The stability form and its family tangent}

For a family we use $H_\mu$ for the mass-zero space denoted by
$\mathsf H_0$ for a fixed star.  For a fixed $f$, put
\begin{equation}\label{eq:fixed-relation-mass-space}
 \ell_\mu(\eta)=\int_{B_{R_\mu}}\eta\ud x,\qquad
 H_\mu=\Ker\ell_\mu,\qquad v_\mu=\partial_\mu\rho_\mu.
\end{equation}
By \cref{thm:radial-index},
\begin{equation}\label{eq:fixed-relation-Q}
 n^u_{\rm rad}(\mu)=n^-(\cQ_\mu|_{H_\mu}),\qquad
 \cQ_\mu=L_\mu+E_\mu,\qquad
 E_\mu[\eta]=\int\mathfrak G_\mu|(\Id-T_\mu)\eta|^2\ud x,
\end{equation}
where
\begin{equation}\label{eq:fixed-relation-L}
 \begin{aligned}
 L_\mu[\eta]
 &=\int\frac{\widetilde P_f'(\rho_\mu)}{\rho_\mu}\eta^2\ud x
   -4\pi\langle(-\Delta)^{-1}\eta,\eta\rangle,\\
 T_\mu\eta
 &=\frac{\rho_\mu'}{4\pi r^2\rho_\mu}I_\eta,\qquad
 I_\eta(r)=4\pi\int_0^r s^2\eta(s)\ud s.
 \end{aligned}
\end{equation}
The effective potential $\Psi_f$ is defined by
$\Psi_f''=\widetilde P_f'/\rho$.  Along the branch,
\begin{equation}\label{eq:fixed-relation-tangent-identity}
 \Psi_f'(\rho_\mu)+V_\mu+c_\mu=0,\qquad
 L_\mu(v_\mu,\eta)=-c_\mu'\ell_\mu(\eta),\qquad
 \ell_\mu(v_\mu)=M_\mu',\qquad c_\mu=\frac{M_\mu}{R_\mu}.
\end{equation}
With the notation of \eqref{eq:fixed-relation-mass-space}, the
identity follows by differentiating the steady equation in the positive
interior and then passing to the weighted spaces using $\mathrm{(F)}$.  The
moving endpoint contributes nothing because $\rho_\mu(R_\mu)=0$.  Crucially,
the normal identity is for $L_\mu$, not for the full form $\cQ_\mu$.

\subsection{Failure of the turning-point principle}

\begin{theorem}[A mass extremum is not a stability transition]
\label{thm:fixed-relation-no-transition}
Let \eqref{eq:fixed-relation-family} be a $C^2$ compact branch satisfying
$\mathrm{(P)}$, $\mathrm{(E_s)}$, and $\mathrm{(F)}$, with the same fixed
function $f$ for every $\mu$.  Suppose $\mu_*\in\operatorname{int}I$ is a
simple mass extremum and
\begin{equation}\label{eq:fixed-relation-extremum-hypotheses}
 M_{\mu_*}'=0,\qquad M_{\mu_*}''\ne0,\qquad
 c_{\mu_*}'\ne0,\qquad
 n^-(L_{\mu_*})=1,\qquad \Ker L_{\mu_*}=\{0\}.
\end{equation}
Then $\cQ_{\mu_*}|_{H_{\mu_*}}$ is coercive.  It remains coercive for
$\mu$ in a neighborhood of $\mu_*$; hence there is no growing radial mode
there and $\mu_*$ is not a radial stability transition.  In particular, if
$\mu_*$ is the first simple mass extremum of the branch, the first
turning-point principle fails.  
\end{theorem}

\begin{proof}
At $\mu_*$, \eqref{eq:fixed-relation-tangent-identity} has the form required
by \cref{lem:one-constraint-inertia}, with $L_{\mu_*}$ from
\eqref{eq:fixed-relation-L} and the normal value \eqref{eq:q-normal} given
by
\[
 L_{\mu_*}[v_{\mu_*}]=-c_{\mu_*}'M_{\mu_*}'=0.
\]
The last two assumptions in
\eqref{eq:fixed-relation-extremum-hypotheses} therefore imply
\begin{equation}\label{eq:fixed-relation-constrained-L-at-extremum}
 L_{\mu_*}|_{H_{\mu_*}}\geq0,\qquad
 \Ker(L_{\mu_*}|_{H_{\mu_*}})=\Span\{v_{\mu_*}\}.
\end{equation}
Here the kernel is the radical of the restricted form.

The Schwarzschild term is strictly positive on this one-dimensional kernel.
Indeed, if $E_{\mu_*}[v_{\mu_*}]=0$, then
$\mathfrak G_{\mu_*}>0$ almost everywhere and
$(\Id-T_{\mu_*})v_{\mu_*}=0$.  With $J=I_{v_{\mu_*}}$, this gives
\[
 J'=\frac{\rho_{\mu_*}'}{\rho_{\mu_*}}J,
 \qquad (J/\rho_{\mu_*})'=0.
\]
Since $J(0)=0$ and $\rho_{\mu_*}(0)=\mu_*>0$, one obtains $J=0$ and then
$v_{\mu_*}=0$, contradicting
$v_{\mu_*}(0)=\partial_\mu\rho_\mu(0)|_{\mu=\mu_*}=1$.
Thus $E_{\mu_*}[v_{\mu_*}]>0$.

Both summands of $\cQ_{\mu_*}=L_{\mu_*}+E_{\mu_*}$ are nonnegative on
$H_{\mu_*}$.  If $\cQ_{\mu_*}[\eta]=0$, then
\eqref{eq:fixed-relation-constrained-L-at-extremum} yields
$\eta\in\Span\{v_{\mu_*}\}$, while the strict inequality just proved forces
$\eta=0$.  Hence the restricted form is nonnegative and has trivial kernel.
By $\mathrm{(F)}$ it is coercive plus compact, so the Fredholm alternative
upgrades this to strict coercivity.  Norm continuity on the fixed pulled-back
space preserves that coercivity for all nearby $\mu$.  The radial index
formula \eqref{eq:fixed-relation-Q} completes the proof.
\end{proof}

\begin{remark}[Scope of the fixed-relation conclusion]
\label{rem:fixed-relation-scope}
The theorem is local to a continued branch satisfying
\eqref{eq:fixed-relation-extremum-hypotheses}.  Low-central-density
existence does not itself supply such an extremum, and the proof does not
assert that instability occurs later on the branch.  Its lifting argument
is more general: on any mass-zero space where $L\geq0$, the relation
$(\Id-T)\eta=0$ gives $I_\eta=C\rho$ and then $C=0$ at the center.
Consequently the positive entropy term removes every possible null
direction of $L$; the Fredholm property upgrades positivity of $L+E$
to coercivity.
\end{remark}

\section{Fixed entropy distribution: existence and the turning-point principle}
\label{sec:fixed-lagrangian}

This second construction is perturbative.  In addition to $\mathrm{(P)}$, assume
\begin{equation}\label{eq:lagrangian-factorized-pressure}
 \begin{gathered}
 P(S,\rho)=e^Sp(\rho),\qquad
 p(\rho)=\rho^{\gamma_0}a(\rho),\\
 a\in C^4([0,\mu_++1]),\qquad
 \inf a>0,\qquad p'>0,\qquad
 \alpha=(\gamma_0-1)^{-1}>1.
 \end{gathered}
\end{equation}
Fix $I=[\mu_-,\mu_+]\Subset(0,\infty)$ on a finite-radius isentropic
branch, defined on a neighborhood of $I$, whose enthalpy has a transverse
first zero.  Assume
\begin{equation}\label{eq:lagrangian-comparison-index-package}
 c_{\mu,0}'=\left(\frac{M_{\mu,0}}{R_{\mu,0}}\right)'\geq c_0>0,
 \qquad n^-(L_{\mu,0})=1,\qquad \Ker L_{\mu,0}=\{0\},
\end{equation}
and that the reference mass has one simple maximum
$\mu_*\in(\mu_-,\mu_+)$:
\begin{equation}\label{eq:lagrangian-comparison-mass-maximum}
 M_{\mu,0}'>0\ (\mu<\mu_*),\quad
 M_{\mu_*,0}'=0,\quad M_{\mu_*,0}''<0,\quad
 M_{\mu,0}'<0\ (\mu>\mu_*).
\end{equation}
These are explicit hypotheses on the comparison branch.  Set
\begin{equation}\label{eq:lagrangian-reference-enthalpy}
 H_0(\rho)=\int_0^\rho\frac{p'(q)}q\ud q,
 \qquad F_0=H_0^{-1},
 \qquad F_0(y)=y^\alpha G_0(y^\alpha),
\end{equation}
where $G_0$ is positive and has the regularity inherited from
\eqref{eq:lagrangian-factorized-pressure}.

Entropy is fixed as a function of physical enclosed mass:
\begin{equation}\label{eq:lagrangian-entropy-family}
 S_{\mu,\kappa}(r)=\sigma(m_{\mu,\kappa}(r)),\qquad
 \sigma(m)=\kappa\overline S(m),\qquad
 m_\rho(r)=4\pi\int_0^r s^2\rho(s)\ud s .
\end{equation}
Take $\overline S\in C^4([0,M_++1])$ with $\overline S'>0$, where
$M_+>\sup_{\mu\in I}M_{\mu,0}$.  The parameter $\kappa>0$ is fixed along
each central-density family.

\subsection{Existence with a fixed entropy distribution}

Hydrostatic balance has the correctly normalized form
\begin{equation}\label{eq:lagrangian-equilibrium-ode}
 m'=4\pi r^2\rho,\qquad
 P_\rho\rho'+4\pi r^2\rho P_S\sigma'(m)=-\rho\frac{m}{r^2},
 \qquad \rho(0)=\mu,\quad m(0)=0.
\end{equation}
Equivalently, where $m>0$ is a coordinate,
\[
 \frac{\ud P}{\ud m}=-\frac{m}{4\pi r^4},\qquad
 \frac{\ud(r^3)}{\ud m}=\frac{3}{4\pi\rho}.
\]

\begin{lemma}[Existence for a fixed entropy distribution]
\label{lem:lagrangian-family-existence}
For all sufficiently small $\kappa\geq0$, \eqref{eq:lagrangian-equilibrium-ode}
has a unique family near the reference branch up to a finite first vacuum.
Its mass and radius are jointly $C^2$ in $(\mu,\kappa)$ and, uniformly on
$I$,
\begin{align}\label{eq:lagrangian-curve-convergence}
 (R_{\cdot,\kappa},M_{\cdot,\kappa})
 =(R_{\cdot,0},M_{\cdot,0})+O(\kappa)\quad\hbox{in }C^1(I),
\end{align}
\begin{align}\label{eq:lagrangian-second-curve-convergence}
 (R_{\cdot,\kappa},M_{\cdot,\kappa})
 \longrightarrow(R_{\cdot,0},M_{\cdot,0})\quad\hbox{in }C^2(I).
\end{align}
Writing
$\widehat\rho_{\mu,\kappa}(s)=\rho_{\mu,\kappa}(R_{\mu,\kappa}s)$,
\begin{align}\label{eq:lagrangian-rescaled-density-ratio}
 \widehat\rho_{\mu,\kappa}(s)=(1-s)^\alpha c_{\rho,\mu,\kappa}(s),
 \qquad \inf_{\mu,\kappa,s}c_{\rho,\mu,\kappa}(s)>0.
\end{align}
The positive factors are continuous in $(\mu,\kappa)$ in
$W^{1,\infty}(0,1)$ and differ from their $\kappa=0$ values by
$O_I(\kappa)$ there.  The zero-extended tangents belong to the density spaces.
For $\kappa>0$ the equilibria satisfy $\mathrm{(E_s)}$.
\end{lemma}

\begin{proof}
Use $y=H_0(\rho)$ and $F_0$ from
\eqref{eq:lagrangian-reference-enthalpy}.  Then
\begin{equation}\label{eq:lagrangian-enthalpy-ode}
 m'=4\pi r^2F_0(y),\qquad
 y'=-D_\kappa(r,m,y),\qquad
 D_\kappa=e^{-\kappa\overline S(m)}\frac{m}{r^2}
       +4\pi\kappa r^2p(F_0(y))\overline S'(m).
\end{equation}
The center is handled by the Volterra system
\begin{equation}\label{eq:lagrangian-center-Volterra}
 \begin{split}
 m(r)&=4\pi\int_0^r t^2F_0(y(t))\ud t,\\
 y(r)&=H_0(\mu)-\int_0^r D_\kappa(t,m(t),y(t))\ud t.
 \end{split}
\end{equation}
On a fixed small interval, uniformly for $\mu\in I$, $y$ stays in a
compact positive enthalpy interval.  Substitution of the first equation
into the second gives a contraction on continuous $y$: differences in
$m$ are bounded by $Cr^3\|\delta y\|_\infty$, and integration of the
$m/r^2$ term contributes $Cr^2\|\delta y\|_\infty$.  The remaining term
contributes $C|\kappa|r^3\|\delta y\|_\infty$.
The contraction and its first two differentiated equations have uniform
bounds, since $F_0$ is smooth on this positive interval.  Thus the center
solution is jointly $C^2$ in $(\mu,\kappa)$, with $m=O(r^3)$.
Ordinary interior continuation preserves this regularity while $y>0$.

We next make the passage to the surface on a fixed domain.  Choose a
single $y_0>0$ smaller than every $H_0(\mu)$, sufficiently small that
the reference crossings $y(r_e;\mu,0)=y_0$ lie in surface intervals on
which $r,m,D_0=m/r^2$ have uniform positive lower bounds.  Such a choice
is possible by compactness of $I$ and transversality of the first zeros.
At this positive level $y_r=-D_\kappa$ stays bounded away from zero.
The implicit-function theorem gives jointly $C^2$ entrance data
$r_e(\mu,\kappa)$ and $m_e(\mu,\kappa)=m(r_e;\mu,\kappa)$.
For parameter indices $i,j\in\{\mu,\kappa\}$ the entrance derivatives are
\begin{equation}\label{eq:lagrangian-entrance-derivatives}
 (r_e)_i=-\frac{y_i}{y_r},\qquad
 (r_e)_{ij}=-\frac{y_{ij}+y_{ri}(r_e)_j+y_{rj}(r_e)_i
                      +y_{rr}(r_e)_i(r_e)_j}{y_r},
\end{equation}
evaluated at the crossing.  The corresponding formulas for $m_e$ follow
by the chain rule.  In particular, all entrance derivatives are uniformly
bounded.

Use $y$ as the independent variable thereafter:
\begin{equation}\label{eq:fixed-enthalpy-lagrangian-ode}
 r_y=-D_\kappa^{-1},\qquad
 m_y=-4\pi r^2F_0(y)D_\kappa^{-1},\qquad 0\leq y\leq y_0.
\end{equation}
Write this as $Z_y=\mathcal F(y,Z,\kappa)$, $Z=(r,m)$, with
$Z(y_0)=(r_e,m_e)$.  In a uniform neighborhood of the reference solutions,
$r,m,D_\kappa$ are bounded above and below by positive constants.
The function $\mathcal F$ is continuous in $y\in[0,y_0]$ and has bounded
continuous derivatives through order two in $(Z,\kappa)$ there.
With explicit parameter derivatives of $\mathcal F$ denoted by subscripts,
the differentiated integral equations are
\begin{align}
 Z_i(y)&=(Z_e)_i+\int_{y_0}^y
          (\mathcal F_Z Z_i+\mathcal F_i)(t)\ud t,
          \label{eq:lagrangian-first-parameter-system}\\
 Z_{ij}(y)&=(Z_e)_{ij}+\int_{y_0}^y
  \bigl(\mathcal F_Z Z_{ij}+\mathcal F_{ZZ}[Z_i,Z_j]
       +\mathcal F_{Zi}Z_j+\mathcal F_{Zj}Z_i+\mathcal F_{ij}\bigr)(t)\ud t.
          \label{eq:lagrangian-second-parameter-system}
\end{align}
Here $\mathcal F_\mu=0$; dependence on $\mu$ enters through the entrance
data.  Gronwall's inequality gives uniform bounds and continuous
parameter dependence for these equations up to $y=0$.  The first vacuum
radius and mass are $R= r(0)$ and $M=m(0)$.  This proves joint $C^2$
regularity at the moving surface without requiring $F_0''(0)$ to exist.
Bounded mixed derivatives give the $O(\kappa)$ estimate in $C^1(I)$;
uniform continuity of the second derivatives gives convergence in $C^2(I)$.
These are \eqref{eq:lagrangian-curve-convergence} and
\eqref{eq:lagrangian-second-curve-convergence}.

We also record the spatial estimates needed for the density factor.
The factorization in \eqref{eq:lagrangian-reference-enthalpy} implies
$F_0'\in L^\infty$ and
$\partial_y^2[p(F_0(y))]=O(y^{\alpha-1})$ near zero.
Differentiating \eqref{eq:fixed-enthalpy-lagrangian-ode} therefore bounds
$r_y,r_{yy},r_{yyy},r_{yi},r_{yyi}$, uniformly for $i=\mu,\kappa$;
only $F_0'$ and the displayed second derivative of $p\circ F_0$ occur.
Moreover, $|r_y|$ is bounded below.  Inverting $r(y)$ and combining with
the interior estimates shows that
$\widehat y(s)=y(R_{\mu,\kappa}s)$ is continuous in parameters in
$W^{2,\infty}(0,1)$ and differs from its $\kappa=0$ value by
$O_I(\kappa)$ in this norm.  Since $\widehat y(1)=0$,
\begin{equation}\label{eq:lagrangian-rescaled-enthalpy-factor}
 \widehat y(s)=(1-s)b_{\mu,\kappa}(s),\qquad
 b_{\mu,\kappa}(s)=\int_0^1
       -\widehat y'\bigl(s+t(1-s)\bigr)\ud t.
\end{equation}
The integral formula and its $s$ derivative show
\[
 \|b_{\mu,\kappa}-b_{\nu,\lambda}\|_{W^{1,\infty}}
 \leq C\|\widehat y_{\mu,\kappa}-\widehat y_{\nu,\lambda}\|_{W^{2,\infty}}.
\]
Transversality gives a uniform positive lower bound for $b$ near $s=1$;
positivity of $\widehat y$ and compactness give it on the remaining
interval.  Consequently
\[
 c_{\rho,\mu,\kappa}(s)=b_{\mu,\kappa}(s)^\alpha
 G_0\bigl((1-s)^\alpha b_{\mu,\kappa}(s)^\alpha\bigr)
\]
has all the properties in \eqref{eq:lagrangian-rescaled-density-ratio}.

For clarity, the tangent is differentiated at fixed physical radius.
In the surface coordinates, $y_i=-r_i/r_y$ is uniformly bounded, so
\begin{equation}\label{eq:lagrangian-weighted-tangent}
 v_{\mu,\kappa}=F_0'(y)y_\mu=O_I((R-r)^{\alpha-1}),\qquad
 \int_{R-\eps}^R w_1|v_{\mu,\kappa}|^2r^2\ud r
 \leq C_I\eps^\alpha.
\end{equation}
Here $w_1\asymp(R-r)^{1-\alpha}$, as also checked below.
Interior regularity and this estimate prove membership in the weighted
density space.  The zero extension has no surface delta term because
$\rho(R)=0$; the same estimates justify
$M_\mu'=4\pi\int_0^Rr^2v_{\mu,\kappa}\ud r$.
These are first-tangent assertions, not claims of twice differentiable
zero extensions in the weighted norm.

For $\kappa>0$, $D_\kappa>0$ implies $\rho'<0$.  The induced
entropy--density derivative and Schwarzschild coefficient are
\[
 f_\mu'(\rho(r))=\frac{\kappa\overline S'(m)4\pi r^2\rho}{\rho'}<0,
 \qquad \mathfrak G=-\frac{P_Sf_\mu'}{\rho}>0\quad(0<r<R).
\]
At the center, writing $P_{\rho,c}=P_\rho(\kappa\overline S(0),\mu)$ and
$P_{S,c}=P_S(\kappa\overline S(0),\mu)$, one obtains
\[
 \rho'=-\frac{4\pi\mu^2}{3P_{\rho,c}}r+O(r^2),\qquad
 \mathfrak G=
 \frac{3\kappa P_{S,c}P_{\rho,c}\overline S'(0)}{\mu^2}r
 +O_I(\kappa r^2).
\]
Thus $\mathfrak G\asymp\kappa r$ and
$\int_{|x|<\delta}\mathfrak G^{-1}\ud x<\infty$.
At vacuum, \eqref{eq:lagrangian-enthalpy-ode} gives
$-f_\mu'=O_I(\kappa\rho^{\gamma_0-1})$ and
$\mathfrak G=O_I(\kappa\rho^{2\gamma_0-2})$.
The effective weight is
$w_1=P_\rho/\rho-\mathfrak G\asymp\rho^{\gamma_0-2}$, and is positive
also directly by hydrostatic balance.  These estimates verify the endpoint
regularity of $f_\mu$, the weight used in
\eqref{eq:lagrangian-weighted-tangent}, and all conditions in
$\mathrm{(E_s)}$.
\end{proof}

\subsection{The full Hessian and the family tangent}

The mass prescription is inserted before taking variations.  For radial
densities define
\begin{equation}\label{eq:lagrangian-reduced-energy}
 \cE_\sigma[\rho]=\int\Phi(\sigma(m_\rho),\rho)\ud x
                 -\frac1{8\pi}\int|\nabla V_\rho|^2\ud x,
 \qquad J_\eta(r)=4\pi\int_0^r s^2\eta(s)\ud s .
\end{equation}
Variations are first taken in the smooth interior core; the resulting
bilinear forms are extended to the weighted density space.
The pullback \eqref{eq:mass-preserving-pullback} identifies these spaces with
\begin{equation}\label{eq:lagrangian-fixed-reference-space}
 \mathsf X=\bigl(L^2_{(1-|x|)^{1-\alpha}}(B_1)\bigr)^{\rm rad},
 \qquad
 \mathsf H=\left\{\eta\in\mathsf X:\int_{B_1}\eta\ud x=0\right\}.
\end{equation}

\begin{proposition}\label{prop:lagrangian-hessian}
For the family in \cref{lem:lagrangian-family-existence}, the full Hessian is
\begin{equation}\label{eq:lagrangian-hessian}
\begin{aligned}
 \cQ_\mu^{\rm Lag}[\eta]
 ={}&\int\Phi_{\rho\rho}\eta^2\ud x
       -4\pi\langle(-\Delta)^{-1}\eta,\eta\rangle\\
 &+2\int B_\mu J_\eta\eta\ud x+\int C_\mu^{\rm Lag} J_\eta^2\ud x,\\
 B_\mu={}&\Phi_{\rho S}\sigma'(m_\mu),\qquad
 C_\mu^{\rm Lag}=\Phi_{SS}|\sigma'(m_\mu)|^2+\Phi_S\sigma''(m_\mu).
\end{aligned}
\end{equation}
The thermodynamic derivatives here are evaluated at $(S_\mu,\rho_\mu)$.
On $H_\mu=\{\eta:\int\eta=0\}$ this is the dynamically accessible form.
On the full radial density space it satisfies
\begin{equation}\label{eq:lagrangian-tangent-identity}
 \cQ_\mu^{\rm Lag}(v_\mu,\eta)=-c_\mu'\ell_\mu(\eta),\qquad
 v_\mu=\partial_\mu\rho_\mu,\qquad c_\mu=\frac{M_\mu}{R_\mu}.
\end{equation}
On the fixed space \eqref{eq:lagrangian-fixed-reference-space} the forms are
norm-continuous and
\begin{equation}\label{eq:lagrangian-form-convergence}
 \|\widehat\cQ_{\mu,\kappa}^{\rm Lag}
           -\widehat L_{\mu,0}\|_{\mathcal B(\mathsf X,\mathsf X^*)}
 \leq C_I\kappa .
\end{equation}
They satisfy $\mathrm{(F)}$ for small $\kappa$.
\end{proposition}

\begin{proof}
Since $\delta m_\rho=J_\eta$, the first variation of
\eqref{eq:lagrangian-reduced-energy} is
\[
 \delta\cE_\sigma[\rho](\eta)
 =\int\left[\Phi_\rho+V_\rho+
       4\pi\int_r^R\Phi_S(s)\sigma'(m_\rho(s))s^2\ud s\right]\eta(r)\ud x .
\]
This follows by reversing the order of integration in
$\int\Phi_S\sigma'J_\eta\ud x$.  The derivative in $r$ of the bracket is
$P'/\rho+V_\rho'$, by \eqref{eq:pressure-gradient-identity}, so it is
constant along an equilibrium.  Evaluation at vacuum gives
\begin{equation}\label{eq:lagrangian-Euler-tail}
 \Phi_\rho+V_\mu+
 4\pi\int_r^{R_\mu}\Phi_S(s)\sigma'(m_\mu(s))s^2\ud s+c_\mu=0,
 \qquad c_\mu=\frac{M_\mu}{R_\mu}.
\end{equation}
A second variation, using
$\delta S=\sigma'J_\eta$ and $\delta^2 S=\sigma''J_\eta^2$,
gives \eqref{eq:lagrangian-hessian}.

To compare with the accessible form, use
$\Phi_{\rho S}-P_S/\rho=\Phi_S/\rho$ and
\[
 \left(\frac{\Phi_S\sigma'}{\rho}\right)'
 =\frac{P_S\sigma'\rho'}{\rho^2}+4\pi r^2 C_\mu^{\rm Lag}.
\]
For a smooth mass-zero variation, $J'=4\pi r^2\eta$ and
$J(0)=J(R)=0$.  Integration by parts therefore gives
\[
 2\int B_\mu J\eta\ud x+\int C_\mu^{\rm Lag} J^2\ud x
 =2\int\frac{P_S\sigma'}{\rho}J\eta\ud x
  -\int_0^R\frac{P_S\sigma'\rho'}{\rho^2}J^2\ud r .
\]
Using $f_\mu'=\sigma'4\pi r^2\rho/\rho'$ and
$\mathfrak G=-P_Sf_\mu'/\rho$, this is exactly the expansion of
$\int[\Psi''\eta^2+\mathfrak G|\eta-T_\mu\eta|^2]\ud x
-\int\Phi_{\rho\rho}\eta^2\ud x$.
Boundedness and the mass-zero density argument in
\cref{prop:radial-range} extend the equality to $H_\mu$.

For completeness, all full-space extensions and parameter limits follow
from
\[
 |J_\eta(r)|\leq C_I\|\eta\|_{X_{1,\mu,\kappa}},\qquad
 |B_\mu|\leq C_I\kappa\rho^{\gamma_0-1},\qquad
 |C_\mu^{\rm Lag}|\leq C_I\kappa\rho^{\gamma_0}.
\]
The first bound is weighted Cauchy--Schwarz, since
$\int w_1^{-1}\ud x$ is uniformly finite.  The other two follow from
\eqref{eq:Phi-definition}.  In particular,
\[
 \left|\int \rho^{\gamma_0-1}J_\eta\eta\ud x\right|
 \leq
 \left(\int\rho^{\gamma_0-2}\eta^2\ud x\right)^{1/2}
 \left(\int\rho^{\gamma_0}J_\eta^2\ud x\right)^{1/2}
 \leq C_I\|\eta\|^2 .
\]
Here is the complete pullback, including the moving-radius factors.
For $h=\mathcal U_{\mu,\kappa}\eta$, put
$J_h(s)=4\pi\int_0^s t^2h(t)\ud t$ and evaluate hatted coefficients at
$r=Rs$, where $R=R_{\mu,\kappa}$.  Then
\begin{equation}\label{eq:lagrangian-full-pullback}
 \begin{split}
 \widehat\cQ_{\mu,\kappa}^{\rm Lag}[h]
 ={}&R^{-3}\int_{B_1}\widehat{\Phi_{\rho\rho}}h^2\ud x
      -4\pi R^{-1}\langle(-\Delta)^{-1}h,h\rangle\\
 &+2\int_{B_1}\widehat B_\mu J_hh\ud x
      +R^3\int_{B_1}\widehat C_\mu^{\rm Lag}J_h^2\ud x .
 \end{split}
\end{equation}
Indeed, $J_\eta(Rs)=J_h(s)$, and the Newton kernel has homogeneity $-1$.
Write $d=1-|x|$ and $w_*=d^{1-\alpha}$.  The positive density factor and
\eqref{eq:lagrangian-factorized-pressure} give, uniformly on $I$,
\begin{equation}\label{eq:lagrangian-pullback-coefficient-bounds}
 \begin{gathered}
 \left\|\frac{R_{\mu,\kappa}^{-3}\widehat{\Phi_{\rho\rho}}_{\mu,\kappa}
       -R_{\mu,0}^{-3}\widehat{\Phi_{\rho\rho}}_{\mu,0}}{w_*}
       \right\|_\infty\leq C_I\kappa,\\
 |\widehat B_\mu|\leq C_I\kappa d,\qquad
 |\widehat C_\mu^{\rm Lag}|\leq C_I\kappa d^{\alpha+1},\qquad
 \|J_h\|_\infty\leq C_I\|h\|_{\mathsf X}.
 \end{gathered}
\end{equation}
For example,
$\Phi_{\rho\rho}=e^S\rho^{\gamma_0-2}
(\gamma_0a(\rho)+\rho a'(\rho))$; dividing its pullback by $w_*$ leaves
a bounded function of the positive factor $c_{\rho,\mu,\kappa}$,
$S$, and $R$, each varying by $O_I(\kappa)$.
Weighted Cauchy--Schwarz bounds the cross term by
\[
 C_I\kappa\|h\|_{\mathsf X}
       \left(\int_{B_1}d^{\alpha+1}J_h^2\ud x\right)^{1/2}
 \leq C_I\kappa\|h\|_{\mathsf X}^2.
\]
The last term of \eqref{eq:lagrangian-full-pullback} has the same bound.
The Newton form is bounded and compact on $\mathsf X$, and
$R_{\mu,\kappa}^{-1}-R_{\mu,0}^{-1}=O_I(\kappa)$.
Together these estimates prove \eqref{eq:lagrangian-form-convergence};
polarization gives the stated bilinear operator norm.  Taking coefficient
differences in the same bounds proves norm continuity in both parameters.
The local multiplication form is uniformly coercive on $\mathsf X$;
adding the two $O(\kappa)$ entropy terms preserves coercivity for small
$\kappa$.  Subtracting the compact Newton form proves $\mathrm{(F)}$.

Finally differentiate the bracket $\mathcal F_\mu$ in
\eqref{eq:lagrangian-Euler-tail} at fixed $r$:
\[
 \partial_\mu\mathcal F_\mu
 =\Phi_{\rho\rho}v_\mu+B_\mu J_{v_\mu}
  -4\pi(-\Delta)^{-1}v_\mu
  +4\pi\int_r^{R_\mu}(B_\mu v_\mu+C_\mu^{\rm Lag} J_{v_\mu})s^2\ud s .
\]
There is no moving-endpoint term since $\Phi_S(S,0)=0$.
The preceding estimates and the weighted tangent bound justify the
differentiation and reversal of integration.  Pairing with $\eta$ gives
$\cQ_\mu^{\rm Lag}(v_\mu,\eta)=-c_\mu'\ell_\mu(\eta)$.
Also $\ell_\mu(v_\mu)=M_\mu'$ because the density vanishes at the moving
surface.  Thus the normal identity holds for the full Hessian.
\end{proof}

\subsection{The turning-point principle}

\begin{theorem}[Turning-point principle near a prescribed isentropic branch]
\label{thm:lagrangian-small-entropy-tpp}
Under the hypotheses of \cref{lem:lagrangian-family-existence}, assume also
\eqref{eq:lagrangian-comparison-index-package} and
\eqref{eq:lagrangian-comparison-mass-maximum}.  For all sufficiently small fixed $\kappa>0$,
the family \eqref{eq:lagrangian-entropy-family} exists and has a unique
simple mass maximum $\mu_*^\kappa\to\mu_*$.  Its unstable dimension is
\begin{equation}\label{eq:lagrangian-unstable-count}
 n^u_{\rm rad}(\mu)=
 \begin{cases}
 0,&M_{\mu,\kappa}'\geq0,\\
 1,&M_{\mu,\kappa}'<0.
 \end{cases}
\end{equation}
The constrained form has a one-dimensional kernel at $\mu_*^\kappa$ and
no kernel elsewhere in $I$.  Thus the stars have no growing radial mode
through the maximum and exactly one after it; the first radial stability
transition in $I$ is precisely that maximum.
\end{theorem}

\begin{proof}
Existence and $C^2$ convergence are given by
\cref{lem:lagrangian-family-existence}.  Choose a small interval
$J\Subset\operatorname{int}I$ containing $\mu_*$ and $\delta>0$ so that
$M_{\mu,0}''\leq-\delta$ on $J$.  On the compact complement of $J$,
$|M_{\mu,0}'|$ has a positive lower bound.  For small $\kappa$, the
$C^2$ convergence preserves strict decrease of $M_{\mu,\kappa}'$ on
$J$ and its opposite signs at the two ends, while $C^1$ convergence
preserves its sign outside $J$.  Hence there is exactly one simple
maximum $\mu_*^\kappa$, and $\mu_*^\kappa\to\mu_*$.
The same convergence gives $c_{\mu,\kappa}'>c_0/2$.
On the fixed space $\mathsf X$, the reference Fredholm operators are
invertible and depend continuously in norm on $\mu$.  Compactness of $I$
therefore gives a uniform bound for their inverses, and hence a uniform
gap from zero.  Equation \eqref{eq:lagrangian-form-convergence} implies
\[
 n^-(\cQ_{\mu,\kappa}^{\rm Lag})=1,\qquad
 \Ker\cQ_{\mu,\kappa}^{\rm Lag}=\{0\}.
\]
Apply \cref{lem:one-constraint-inertia} to
\eqref{eq:lagrangian-tangent-identity}.  Its normal quadratic value is
\[
 \cQ_{\mu,\kappa}^{\rm Lag}[v_{\mu,\kappa}]
 =-c_{\mu,\kappa}'M_{\mu,\kappa}'.
\]
The lemma gives constrained index zero when $M'\geq0$, index one when
$M'<0$, and
$\Ker(\cQ^{\rm Lag}|_{H})=\Span\{v_{\mu,\kappa}\}$ exactly when $M'=0$.
The full Hessian agrees on $H$ with the accessible form, so
\cref{thm:radial-index} proves \eqref{eq:lagrangian-unstable-count}.
At the maximum this is mode stability with a neutral direction, not
coercivity.  The kernel and index statements give the turning-point
principle in the sense of \cref{def:TPP}.
\end{proof}

\begin{remark}
The reference branch, including its simple maximum and the full-space
index assumptions, is prescribed in
\eqref{eq:lagrangian-comparison-index-package}--%
\eqref{eq:lagrangian-comparison-mass-maximum}; it is not constructed from
the vacuum exponent alone.  The comparison as $\kappa\downarrow0$ takes
place in the full density forms on $\mathsf X$.  It does not identify the
entropy phase spaces $X_2=L^2_{\mathfrak G}$ at $\kappa>0$ with a
nondegenerate entropy space at $\kappa=0$.
\end{remark}

\section{Schwarzschild instability}\label{sec:convective}

The Schwarzschild-unstable case is treated through a second-order equation
for the velocity perturbation.  We first construct the relevant operator from
a closed symmetric form.  This avoids any appeal to a formal differential
expression containing singular first-order terms.

\subsection{The admissible unstable sign and the velocity form}

Let \((\rho_0,S_0,0)\) be a compactly supported radial equilibrium with
the regularity specified in $\mathrm{(U)}$ below.  Write
\(S_0=f(\rho_0)\), with \(\rho_0'<0\) on \((0,R)\).  Hydrostatic balance
gives
\begin{equation}\label{eq:effective-pressure-forced-positive}
 \widetilde P'(\rho_0(r))\rho_0'(r)
 =-\rho_0(r)V_0'(r),
 \qquad
 V_0'(r)=\frac{4\pi}{r^2}\int_0^r s^2\rho_0(s)\ud s>0.
\end{equation}
Hence, by \eqref{eq:effective-pressure-forced-positive}, every
such decreasing compactly supported equilibrium necessarily has
\begin{equation}\label{eq:effective-pressure-positive-unstable}
 \widetilde P'(\rho_0)>0.
\end{equation}
This is the effective-pressure positivity
\eqref{eq:effective-pressure-positive} on the decreasing branch.
In particular, a branch with \(\widetilde P'<0\) is incompatible with the
equilibrium class considered here and is not part of the theorem.

Assume
\begin{equation}\label{eq:convective-sign-package}
\begin{gathered}
 P_\rho(S_0,\rho_0)>0,\\
 \mathfrak G(r)=-\frac{P_S(S_0,\rho_0)f'(\rho_0)}{\rho_0}<0
 \quad\hbox{on a nonempty open set }U\Subset B_R.
\end{gathered}
\end{equation}
Equation \eqref{eq:classical-discriminant-relation} gives the exact
relation to the classical discriminant.  In particular, $\mathfrak G$
has its sign, and negativity on $U$ is the convectively unstable sign.

Set
\begin{equation}\label{eq:D-a-definition}
 D_u=\Div(\rho_0u),\qquad a_u=\nabla\rho_0\cdot u,
 \qquad \mathfrak a(r)=\frac{P_\rho(S_0,\rho_0)}{\rho_0}>0.
\end{equation}
We define the symmetric velocity form by
\begin{equation}\label{eq:convective-velocity-form}
 \begin{split}
 \mathfrak q[u,v]
 ={}&\int_{B_R}\mathfrak aD_uD_v\ud x
      -\int_{B_R}\mathfrak G(D_ua_v+a_uD_v)\ud x\\
    &+\int_{B_R}\mathfrak Ga_ua_v\ud x
      -4\pi\langle(-\Delta)^{-1}D_u,D_v\rangle.
 \end{split}
\end{equation}
All terms are separately finite on the domain below.  Since
$\Psi''=\mathfrak a-\mathfrak G$ in the positive-density interior,
this agrees with the split expression
$\langle(\Psi''-4\pi(-\Delta)^{-1})D_u,D_v\rangle
+\int\mathfrak G(D_u-a_u)(D_v-a_v)\ud x$
whenever its two terms are separately integrable.  That split expression
is not needed to define the form on its full domain.  The local quadratic
part of \eqref{eq:convective-velocity-form} is
\begin{equation}\label{eq:convective-local-expanded}
 \int_{B_R}\left(
 \mathfrak aD_u^2-2\mathfrak GD_ua_u+\mathfrak Ga_u^2
 \right)\ud x.
\end{equation}

Let
\begin{equation}\label{eq:convective-form-domain}
\begin{gathered}
 Y=L^2_{\rho_0}(B_R;\R^3),
 \qquad
 \mathcal V=\left\{u\in Y:D_u\in L_{\mathfrak a}^2(B_R)\right\},\\
 \|u\|_{\mathcal V}^2
 =\|u\|_Y^2+\int_{B_R}\mathfrak a|D_u|^2\ud x.
\end{gathered}
\end{equation}

\begin{hypothesis}[Unstable-form package \(\mathrm{(U)}\)]
\label{hyp:unstable-form}
The pressure satisfies $\mathrm{(P)}$, and the radial equilibrium density
has the regularity in \eqref{eq:equilibrium-density-regularity}, with
$\rho_0'(0)=0$.  Its entropy satisfies
$S_0\in C^1([0,R])\cap C^2((0,R))$ and $S_0'(0)=0$.
The induced function $f$ belongs to
$C([0,\rho_0(0)])\cap C^2((0,\rho_0(0)))$; no bounded endpoint value of
$f'$ is assumed.  The two-sided physical-vacuum estimates
\eqref{eq:physical-vacuum}--\eqref{eq:physical-vacuum-derivative} hold.
Assume also \eqref{eq:convective-sign-package} and
\begin{equation}\label{eq:convective-coefficient-bound}
 \left(|\mathfrak G|+\frac{\mathfrak G^2}{\mathfrak a}\right)
 \frac{|\nabla\rho_0|^2}{\rho_0}\in L^\infty(B_R).
\end{equation}
The domain in \eqref{eq:convective-form-domain} is the maximal
distributional graph domain.  No separate smooth-core hypothesis is
imposed.
\end{hypothesis}

\begin{lemma}[Closed realization of the velocity form]
\label{lem:convective-closed-form}
Under \(\mathrm{(U)}\), the form \(\mathfrak q\) defined by
\eqref{eq:convective-velocity-form} on \(\mathcal V\) is densely defined,
closed, and lower semibounded on \(Y\).  It determines a unique
self-adjoint operator
\begin{equation}\label{eq:convective-self-adjoint-operator}
 \widetilde L:D(\widetilde L)\subset Y\longrightarrow Y,
 \qquad
 \mathfrak q[u,v]=(\widetilde Lu,v)_Y
 \quad(u\in D(\widetilde L),\ v\in\mathcal V).
\end{equation}
\end{lemma}

\begin{proof}
The graph of the maximal weighted-divergence operator is closed.  Indeed,
suppose \(u_j\to u\) in \(Y\) and \(D_{u_j}\to g\) in
\(L_{\mathfrak a}^2\).  For every
\(\varphi\in C_c^\infty(B_R)\),
\[
 \langle g,\varphi\rangle
 =\lim_{j\to\infty}\langle D_{u_j},\varphi\rangle
 =-\lim_{j\to\infty}\int\rho_0u_j\cdot\nabla\varphi\ud x
 =-\int\rho_0u\cdot\nabla\varphi\ud x.
\]
Thus \(D_u=g\) distributionally, so \(u\in\mathcal V\).
Moreover, $C_c^\infty(B_R;\R^3)\subset\mathcal V$ is dense in $Y$.
Here $\mathfrak a^{-1}$ is bounded on $B_R$: this follows from positivity
and continuity in the interior, and from
$\mathfrak a\asymp(R-r)^{1-\alpha}$ near vacuum, with
$\alpha=(\gamma_0-1)^{-1}>1$.  Thus all distributional pairings used
above are justified.  Consequently the positive reference form
\[
 \mathfrak q_0[u]=\|u\|_Y^2+\int \mathfrak a|D_u|^2\ud x
\]
is densely defined and closed on \(\mathcal V\).

By \eqref{eq:convective-coefficient-bound},
\begin{align*}
 \int|\mathfrak G|a_u^2\ud x
 &\le C\|u\|_Y^2,\\
 \int\frac{\mathfrak G^2}{\mathfrak a}a_u^2\ud x
 &\le C\|u\|_Y^2.
\end{align*}
Consequently, the mixed term in
\eqref{eq:convective-local-expanded} satisfies, for every \(\eps>0\),
\[
 2\left|\int\mathfrak GD_ua_u\ud x\right|
 \le\eps\int \mathfrak a|D_u|^2\ud x+C_\eps\|u\|_Y^2.
\]
To estimate the Newton term on $\R^3$, first extend $w=\rho_0u$ and
$D_u$ by zero.  There is no surface distribution.  Indeed, for a smooth
test function $\varphi$ and a cutoff $\chi_\eps$ which is one away from
the layer $d=R-r<2\eps$ and vanishes near the boundary,
\[
 \left|\int_{d<2\eps}\rho_0u\cdot\nabla\chi_\eps\,\varphi\ud x\right|
 \leq C_\varphi\|u\|_{Y(\{d<2\eps\})}
                    \eps^{(\alpha-1)/2}\longrightarrow0.
\]
Also $D_u\in L^1(B_R)$ by weighted Cauchy--Schwarz and integrability
of $\mathfrak a^{-1}$.  Passing to the limit in the interior divergence
identity gives
\begin{equation}\label{eq:convective-zero-extension}
 \Div(\mathbf1_{B_R}\rho_0u)=\mathbf1_{B_R}D_u
 \quad\text{in }\mathcal D'(\R^3).
\end{equation}
The Fourier multiplier estimate on the whole space now gives
\[
 \left\langle(-\Delta)^{-1}\Div w,\Div w\right\rangle
 =\|\nabla(-\Delta)^{-1}\Div w\|_{L^2}^2
 \le\|w\|_{L^2}^2
 \le\|\rho_0\|_{L^\infty}\|u\|_Y^2.
\]
Thus all terms other than \(\int \mathfrak a|D_u|^2\) are form-bounded, and
the cross term has arbitrarily small relative bound.  The KLMN theorem,
applied after adding a sufficiently large multiple of the \(Y\)-norm,
proves that \(\mathfrak q\) on \(\mathcal V\) is closed and lower semibounded
\cite[Chap.~VI, Sec.~3]{Kato1995}.  The first representation theorem
\cite[Chap.~VI, Sec.~2]{Kato1995} then gives the unique self-adjoint operator
in \eqref{eq:convective-self-adjoint-operator}.  This construction specifies
the operator by its maximal-domain closed form.  It makes no additional
identification with the Friedrichs extension of a separately prescribed
minimal differential operator.
\end{proof}

\subsection{A negative direction}

\begin{lemma}\label{lem:convective-negative-direction}
Under $\mathrm{(U)}$, there exists
\(u\in\mathcal V\) such that \(\mathfrak q[u]<0\).  Consequently,
\begin{equation}\label{eq:negative-spectral-bottom}
 \eta_0:=\inf\spec(\widetilde L)<0.
\end{equation}
\end{lemma}

\begin{proof}
Choose \(x_0\in U\setminus\{0\}\).  Since \(\rho_0'(|x_0|)<0\), there are
constant vectors \(a,e\in\R^3\) such that
\[
 \nabla\rho_0(x_0)\cdot(a\times e)\ne0.
\]
Take \(\chi\in C_c^\infty(U)\) with \(\chi=1\) near \(x_0\), set
\[
 \psi(x)=\chi(x)\,a\cdot(x-x_0),\qquad
 w=\nabla\times(\psi e)=\nabla\psi\times e.
\]
Then \(w\in C_c^\infty(U;\R^3)\), \(\Div w=0\), and
\[
 \nabla\rho_0(x_0)\cdot w(x_0)
 =\nabla\rho_0(x_0)\cdot(a\times e)\ne0.
\]
Thus \(\nabla\rho_0\cdot w\not\equiv0\).  Since
\(U\Subset B_R\), the function \(u=w/\rho_0\) belongs to \(\mathcal V\).  It
satisfies
\[
 D_u=\Div w=0,\qquad
 a_u=\frac{\nabla\rho_0\cdot w}{\rho_0}.
\]
Therefore every term in \eqref{eq:convective-velocity-form} containing
\(D_u\) vanishes and
\[
 \mathfrak q[u]=\int_U\mathfrak G|a_u|^2\ud x<0.
\]
The variational characterization of the spectral bottom gives
\eqref{eq:negative-spectral-bottom}.
\end{proof}

\subsection{Spectral growth for the wave equation}

The velocity component of the linearized system satisfies
\begin{equation}\label{eq:convective-wave-equation}
 u_{tt}+\widetilde Lu=0.
\end{equation}
For \(C>-\eta_0\), equip \(\mathcal V=D((\widetilde L+C)^{1/2})\) with its
form norm.

\begin{theorem}[Schwarzschild instability]
\label{thm:Schwarzschild-instability}
Assume the unstable-form package \(\mathrm{(U)}\).  Let \(\eta_0\) be defined by
\eqref{eq:negative-spectral-bottom}.  Then:
\begin{enumerate}[label=(\roman*)]
\item Equation \eqref{eq:convective-wave-equation} defines a strongly
continuous group on \(\mathcal V\times Y\), and
\begin{equation}\label{eq:convective-upper-growth}
 \|(u(t),u_t(t))\|_{\mathcal V\times Y}
 \le C_0e^{\sqrt{-\eta_0}|t|}
 \|(u(0),u_t(0))\|_{\mathcal V\times Y}.
\end{equation}
The quadratic energy
\begin{equation}\label{eq:convective-energy}
 \|u_t(t)\|_Y^2+\mathfrak q[u(t)]
\end{equation}
is conserved.
\item For every \(\eps\in(0,-\eta_0)\), there are nonzero data
\(\phi_\eps\in\mathcal V\), with zero initial time derivative, for which
\begin{equation}\label{eq:convective-lower-growth}
 \|u_\eps(t)\|_Y
 \ge\frac12e^{\sqrt{-\eta_0-\eps}\,t}\|\phi_\eps\|_Y,
 \qquad t\ge0.
\end{equation}
On the spectral subspace used to choose \(\phi_\eps\), the \(Y\)- and
\(\mathcal V\)-norms are equivalent, so the right-hand side may equivalently
be written as \(c_\eps e^{\sqrt{-\eta_0-\eps}t}
\|\phi_\eps\|_{\mathcal V}\).
\end{enumerate}
\end{theorem}

\begin{proof}
The spectral theorem for a lower-semibounded self-adjoint operator gives the
cosine and sine families for \eqref{eq:convective-wave-equation}.  On the
negative spectral interval, put $a=\sqrt{-\eta_0}>0$ and
$b=\sqrt{-\lambda}\in[0,a]$.  Then
$ \cosh(b|t|)\leq e^{a|t|},\qquad
 \frac{\sinh(b|t|)}b\leq\frac{\sinh(a|t|)}a$,
where the value at $b=0$ is $|t|$.
These estimates also control the differentiated families in the form-energy
norm, since the negative spectral interval is bounded.  On bounded
nonnegative spectral intervals, use
$|\sin(t\sqrt\lambda)|/\sqrt\lambda\leq|t|$; on the remaining positive
spectrum the factors $\sqrt\lambda$ are controlled by the form norm.
Finally, $1+|t|\leq C_a e^{a|t|}$.  This proves
\eqref{eq:convective-upper-growth}.  Conservation of
\eqref{eq:convective-energy} follows first for data in
\(D(\widetilde L)\times\mathcal V\) and then by density.

For \(0<\eps<-\eta_0\), the spectral projection
$ E_{\widetilde L}([\eta_0,\eta_0+\eps])$
is nonzero by the definition of \(\eta_0\).  Choose a nonzero
\(\phi_\eps\) in its range and take
\(u_\eps(0)=\phi_\eps\), \(u_{\eps,t}(0)=0\).  The interval lies in the
negative axis, and spectral calculus gives
\begin{align*}
 \|u_\eps(t)\|_Y^2
 &=\int_{[\eta_0,\eta_0+\eps]}
 \cosh^2(t\sqrt{-\lambda})
 \ud\|E_{\widetilde L}(\lambda)\phi_\eps\|_Y^2\\
 &\ge\frac14e^{2\sqrt{-\eta_0-\eps}\,t}
 \|\phi_\eps\|_Y^2.
\end{align*}
Taking square roots proves \eqref{eq:convective-lower-growth}.  Because the
spectral interval is bounded, the form norm and \(Y\)-norm are equivalent on
its range.  If \(\eta_0\) is an eigenvalue, one may instead choose an
eigenfunction and obtain the exact rate
\(\cosh(t\sqrt{-\eta_0})\).
\end{proof}

\begin{proposition}[Reconstruction of the first-order variables]
\label{prop:convective-reconstruction}
Let \(u\in C(\R;\mathcal V)\cap C^1(\R;Y)\) be an energy solution of
\eqref{eq:convective-wave-equation} with \(u_t(0)=0\).  Define
\[
 \mathcal X_D=L_{\mathfrak a}^2(B_R)
\]
and, with
\(
 Z=\{x:|\nabla S_0(x)|=0\},
\)
\[
 \mathcal S_f=
 \left\{s:s=0\ \hbox{a.e. on }Z,\quad
 \int_{B_R\setminus Z}
 \frac{\rho_0|s|^2}
 {|\nabla S_0|^2}\ud x<\infty\right\}.
\]
Then
\begin{equation}\label{eq:convective-reconstructed-variables}
 \eta(t)=-\int_0^tD_{u(\tau)}\ud\tau\in C^1(\R;\mathcal X_D),
 \qquad
 s(t)=-\int_0^t u(\tau)\cdot\nabla S_0\ud\tau
       \in C^1(\R;\mathcal S_f).
\end{equation}
The triple \((\eta,s,u)\), with \(\eta(0)=s(0)=0\), is an energy solution of
the first-order system \eqref{eq:linearized-EP}.  More precisely, the
continuity and entropy equations hold in the displayed spaces, and the
integrated weak momentum equation is
\begin{equation}\label{eq:convective-integrated-momentum}
 (u_t(t),v)_Y+\int_0^t\mathfrak q[u(\tau),v]\ud\tau=0,
 \qquad v\in\mathcal V.
\end{equation}
For the growing solutions in
\cref{thm:Schwarzschild-instability}, this reconstruction retains the lower
bound \eqref{eq:convective-lower-growth} in every product norm containing the
\(Y\)-norm of \(u\).
\end{proposition}

\begin{proof}
The definition of \(\mathcal V\) gives
$ \|D_v\|_{\mathcal X_D}\leq\|v\|_{\mathcal V}$.
Moreover,
\[
 \|v\cdot\nabla S_0\|_{\mathcal S_f}^2
 =\int_{B_R\setminus Z}
 \rho_0\frac{|\nabla S_0\cdot v|^2}{|\nabla S_0|^2}\ud x
 \leq\|v\|_Y^2.
\]
The Bochner integrals in
\eqref{eq:convective-reconstructed-variables} are therefore well defined and
have the asserted \(C^1\) regularity.  Their derivatives are precisely the
continuity and entropy equations.

The form equation for the wave solution holds in \(\mathcal V^*\):
\[
 \langle u_{tt},v\rangle_{\mathcal V^*,\mathcal V}
 +\mathfrak q[u,v]=0,\qquad v\in\mathcal V.
\]
Consequently the \(\mathcal V^*\)-valued momentum residual
\[
 \langle\mathcal R(t),v\rangle
 =(u_t(t),v)_Y+\int_0^t\mathfrak q[u(\tau),v]\ud\tau
\]
has zero distributional derivative.  It vanishes at \(t=0\), because
\(u_t(0)=0\), and hence \(\mathcal R(t)=0\), proving
\eqref{eq:convective-integrated-momentum}.  To identify this residual with the original momentum equation, set
$V_\eta=-4\pi(-\Delta)^{-1}\eta$.  The expanded weak equation is
\begin{equation}\label{eq:convective-expanded-momentum}
 \begin{split}
 (u_t,v)_Y
 &-\int_{B_R}\mathfrak a\eta D_v\ud x
  -\langle V_\eta,D_v\rangle
  +\int_{B_R}\mathfrak G\eta a_v\ud x\\
 &-\int_{B_R}\frac{P_S}{\rho_0}s(D_v-a_v)\ud x=0,
 \qquad v\in\mathcal V.
 \end{split}
\end{equation}
Every term is continuous on
$\mathcal X_D\times\mathcal S_f\times\mathcal V$.
For the entropy terms, write
$|s|=|\nabla S_0|\rho_0^{-1/2}S_*$ on $B_R\setminus Z$ with
$\|S_*\|_2=\|s\|_{\mathcal S_f}$, and use
$ \frac{P_S}{\rho_0}|\nabla S_0|
   =|\mathfrak G|\,|\nabla\rho_0|.$
The two bounds in \eqref{eq:convective-coefficient-bound} then give
\begin{align*}
 \left|\int\frac{P_S}{\rho_0}sD_v\ud x\right|
 &\leq C\|s\|_{\mathcal S_f}\|D_v\|_{\mathcal X_D},\\
 \left|\int\frac{P_S}{\rho_0}sa_v\ud x\right|
 &\leq C\|s\|_{\mathcal S_f}\|v\|_Y,\qquad
 \left|\int\mathfrak G\eta a_v\ud x\right|
 \leq C\|\eta\|_{\mathcal X_D}\|v\|_Y.
\end{align*}
The Newton pairing is bounded because $\mathcal X_D$ embeds into
$L^{6/5}(B_R)$, exactly as in the potential estimate of
\cref{lem:potential-compact}.
For smooth interior variables, testing \eqref{eq:linearized-EP} gives
\eqref{eq:convective-expanded-momentum}: use
$\Div v=(D_v-a_v)/\rho_0$ and
$\nabla V_0=-(\mathfrak a-\mathfrak G)\nabla\rho_0$.
Thus its continuous form expression is well defined on the stated spaces.
Substituting \eqref{eq:convective-reconstructed-variables} and
$(P_S/\rho_0)f'=-\mathfrak G$ into
\eqref{eq:convective-expanded-momentum} gives exactly
\eqref{eq:convective-integrated-momentum}.  This proves the momentum
equation for the reconstructed variables, both distributionally in the
interior and with the form-domain interpretation for all $v\in\mathcal V$.
No division by $f'$ and no approximation within a spectral subspace is
required.  Finally, the reconstructed product norm is at least
$\|u(t)\|_Y$, so \eqref{eq:convective-lower-growth} remains valid.
\end{proof}

\section*{Acknowledgments}
Zhiwu Lin was supported in part by the National Natural Science Foundation of
China (No.~12494544).  Yucong Wang was supported in part by the National
Natural Science Foundation of China (No.~12501290), the Furong Young Talents
Program for Science and Technology Innovation of Hunan Province
(No.~2026RC3168), the Natural Science Foundation of Hunan Province
(No.~2025JJ60068), the Scientific Research Fund of the Hunan Provincial
Education Department (No.~25B0150), the 111 Project (No.~D23017), and the
Program for Science and Technology Innovative Research Teams in Higher
Educational Institutions of Hunan Province, China. Hao Zhu was supported in part by the National Key R\&D Program of China (No. 2024YFA1013303).

OpenAI's ChatGPT was used during revision to assist with language editing,
\LaTeX{} reorganization, literature checks, and mathematical consistency
checks.  The authors assume responsibility for all content.

\section*{Data availability}
No datasets were generated or analyzed in this study.
\section*{Conflict-of-Interest Statement:}
The authors declare no conflict of interest.

\end{document}